\documentclass{article}
\usepackage{oldlfont}
\usepackage{enumerate}
\usepackage{amsmath}
\usepackage{amssymb}
\usepackage{amsthm}
\usepackage{mathrsfs}
\usepackage{amscd}
\usepackage{exscale}
\usepackage{latexsym}
\usepackage[all]{xy}
\usepackage{hyperref}
\usepackage{fancyhdr}
\usepackage{layout}
\usepackage{color}
\usepackage{graphicx}
\usepackage{mathtools}
\usepackage{amsfonts}
\usepackage{calc}
\usepackage{tikz}
\usepackage{mathtools}
\usepackage[OT2,T1]{fontenc}
\usepackage[margin=1in]{geometry}

\DeclareMathAlphabet{\mathpzc}{OT1}{pzc}{m}{it}
\DeclareSymbolFont{cyrletters}{OT2}{wncyr}{m}{n}
\DeclareMathSymbol{\Sha}{\mathalpha}{cyrletters}{"58}

\begin{document}

\baselineskip=17pt

\pagestyle{headings}

\numberwithin{equation}{section}

\makeatletter       

\def\@fnsymbol#1{\ensuremath{\ifcase#1\or a\or b\or \ddagger\or
   \mathsection\or \mathparagraph\or \|\or a\or \dagger\dagger
   \or \ddagger\ddagger \else\@ctrerr\fi}}

\makeatother

       \renewcommand{\thefootnote}{\fnsymbol{footnote}}                                             

\makeatletter  
\def\section{\@startsection {section}{1}{\z@}{-5.5ex plus -.5ex         
minus -.2ex}{1ex plus .2ex}{\large \bf}}                                 


\pagestyle{fancy}
\renewcommand{\sectionmark}[1]{\markboth{ #1}{ #1}}
\renewcommand{\subsectionmark}[1]{\markright{ #1}}
\fancyhf{} 
\fancyhead[LE,RO]{\slshape\thepage}
\fancyhead[LO]{\slshape\rightmark}
\fancyhead[RE]{\slshape\leftmark}

\addtolength{\headheight}{0.5pt} 
\renewcommand{\headrulewidth}{0pt} 

\newtheorem{thm}{Theorem}[section]
\newtheorem{mainthm}[thm]{Main Theorem}
\newtheorem*{T}{Theorem 1'}

\newcommand{\ZZ}{{\mathbb Z}}
\newcommand{\GG}{{\mathbb G}}
\newcommand{\Z}{{\mathbb Z}}
\newcommand{\RR}{{\mathbb R}}
\newcommand{\NN}{{\mathbb N}}
\newcommand{\GF}{{\rm GF}}
\newcommand{\QQ}{{\mathbb Q}}
\newcommand{\CC}{{\mathbb C}}
\newcommand{\FF}{{\mathbb F}}

\newtheorem{lem}[thm]{Lemma}
\newtheorem{cor}[thm]{Corollary}
\newtheorem{pro}[thm]{Proposition}
\newtheorem*{proposi}{Proposition \ref{pro:pro63}}
\newtheorem*{thm_notag}{Theorem}
\newtheorem{problem}{Problem}
\newtheorem*{prob1}{Problem 1'}
\newtheorem*{con}{Conjecture}

\newtheorem{proprieta}[thm]{Property}
\newcommand{\pf}{\noindent \textbf{Proof.} \ }
\newcommand{\eop}{${\Box}$  \relax}
\newtheorem{num}{equation}{}

\theoremstyle{definition}
\newtheorem{rem}[thm]{Remark}
\newtheorem{rems}[thm]{Remarks}
\newtheorem{D}[thm]{Definition}
\newtheorem{Not}{Notation}

\newtheorem{Def}{Definition}

\newcommand{\nsplit}{\cdot}
\newcommand{\GGG}{{\mathfrak g}}
\newcommand{\GL}{{\rm GL}}
\newcommand{\SL}{{\rm SL}}
\newcommand{\SP}{{\rm Sp}}
\newcommand{\LL}{{\rm L}}
\newcommand{\Ker}{{\rm Ker}}
\newcommand{\la}{\langle}
\newcommand{\ra}{\rangle}
\newcommand{\PSp}{{\rm PSp}}
\newcommand{\Uni}{{\rm U}}
\newcommand{\GU}{{\rm GU}}
\newcommand{\GO}{{\rm GO}}
\newcommand{\Aut}{{\rm Aut}}
\newcommand{\Alt}{{\rm Alt}}
\newcommand{\Sym}{{\rm Sym}}
\renewcommand{\char}{{\rm char}}

\newcommand{\isom}{{\cong}}
\newcommand{\z}{{\zeta}}
\newcommand{\Gal}{{\rm Gal}}
\newcommand{\SO}{{\rm SO}}
\newcommand{\SU}{{\rm SU}}
\newcommand{\PGL}{{\rm PGL}}
\newcommand{\PSL}{{\rm PSL}}
\newcommand{\loc}{{\rm loc}}
\newcommand{\Sp}{{\rm Sp}}
\newcommand{\PUni}{{\rm PU}}
\newcommand{\Id}{{\rm Id}}
\newcommand{\s}{{\sigma}}
\newcommand{\al}{{\alpha}}
\newcommand{\be}{{\beta}}
\newcommand{\ga}{{\gamma}}
\newcommand{\de}{{\delta}}

\newcommand{\F}{{\mathbb F}}
\renewcommand{\O}{{\cal O}}
\newcommand{\Q}{{\mathbb Q}}
\newcommand{\R}{{\mathbb R}}
\newcommand{\N}{{\mathbb N}}
\newcommand{\E}{{\mathcal{E}}}
\newcommand{\G}{{\mathcal{G}}}
\newcommand{\A}{{\mathcal{A}}}
\newcommand{\C}{{\mathcal{C}}}
\newcommand{\modn}{{\textrm{mod} \hspace{0.1cm} }}
\newcommand{\bmu}{{\textbf \mu}}
\newcommand{\hloc}{{\rm loc}}
\newcommand{\f}{{\mathfrak f}}
\newcommand{\g}{{\mathfrak g}}
\newcommand{\h}{{\mathfrak h}}
\newcommand{\cost}{{\mathfrak c}}

\newcommand\ddfrac[2]{\frac{\displaystyle #1}{\displaystyle #2}}

\vskip 0.5cm

\title{On the division fields of an elliptic curve and \\ an effective bound to the hypotheses of the \\ local-global divisibility}
\author{Roberto Dvornicich\footnote{University of Pisa, Department of Mathematics, Pisa, Italy, e-mail address: roberto.dvornicich@unipi.it},
Laura Paladino\footnote{Corresponding author; ORCID ID: 0000-0003-4758-9775; 
University of calabria, Department of Mathematics and Computer Science, Rende, Italy, e-mail address: laura.paladino@unical.it}}
\date{  }
\maketitle

\vskip 1.5cm

\begin{abstract}
We investigate some aspects of the
$m$-division field $K(\E[m])$, where $\E$ is an elliptic curve defined over a field $K$ with $\char(K)\neq 2,3$
and $m$ is a positive integer.
When $m=p^r$,
with $p\geq 5$ a prime and $r$ a positive integer,  we prove $K(\E[p^r])=K(x_1,\z_p,y_2)$,
where $\{(x_1, y_1),(x_2,y_2)\}$ is a generating system of $\E[p^r]$ and $\z_p$ is a primitive $p$-th root of the unity.
If $\E$ has a $K$-rational point of order $p$, then $K(\E[p^r])=K(\z_{p^r},\sqrt[m_1]{a})$, with
$a\in  K(\z_{p^r})$ and $m_1|p^r$. 
In addition, when $K$ is a number field, we produce an upper bound  to the logarithmic height of the discriminant
of the extension $K(\E[m])/K$, for all $m\geq 3$.
As a consequence, we give an explicit effective version of the hypotheses of the local-global divisibility problem
in elliptic curves over number fields. 
\end{abstract}


\section{Introduction} \label{sec0}

\par Let $\E$ be an elliptic curve defined over a field $K$ with $\textrm{char}(K)\neq 2,3$. 
Let $m$ be a positive integer. By $K(\E[m])/K$ we denote the $m$-th division field, i.e. the field generated over
$K$ by the coordinates of the $m$-torsion points of $\E$.
Since the first serious study of the arithmetic of elliptic curves, the $m$-th division fields
have played a vital r\^ole. Investigating the Galois representations on the total Tate module is the same as studying these fields. 
Iwasawa theory, modularity and even the proof of the Mordell-Weil theorem are related to the properties of $K(\E[m])/K$. 
Let $P_1=(x_1,y_1)$ and $P_2=(x_2,y_2)$ be two $m$-torsion points of $\E$, generating $\E[m]$. Then
clearly $K(\E[m])=K(x_1,x_2,y_1,y_2)$. It is well-known that as a consequence of the Weil pairing we have $K(\z_m)\subseteq K(\E[m])$,
where $\z_m$ is a primitive $m$-th root of the unity. Recently there has been a growing interest in studying the possible cases
when $K(\z_m)=K(\E[m])$ \cite{MS}, \cite{Reb}, \cite{Pal2}, \cite{LRGJ} and more generally in producing
new set of generators for $K(\E[m])/K$ involving $\z_m$ itself as a generator \cite{BP0}, \cite{BP}. 
 In \cite{BP}, in particular it was proved that,  for every odd integer $m\geq 5$, we have $K(\E[m])=K(x_1,\z_{m},y_2)$. When $m=p^r$, for a prime $p\geq 5$ and a positive integer $r\geq 2$, we improve such a generating system, by showing $K(\E[p^r])=K(x_1,\z_p,y_2)$, for every $r\geq 1$. 
Even if we  substitute $\z_{p^r}$ with $\z_p$, we still have that all the $p^r$-th roots of the unity are
contained in $K(\E[p^r])$.  We also show 
that, for every $p^r$-torsion
point $(x_1,y_1)$ of $\E$, the extension $K(\E[p^r])/K(x_1,y_1)$ is  metacyclic. Let $C_m$ denote the cyclic group of order $m$. If $F=K(x_1,y_1)$, then  
$K(\E[p^r])/F=F(\z_{p^r},\sqrt[m_1]{a}),$ where $a\in F(\z_{p^r})$, $m_1|p^r$
and $\Gal(K(\E[p^r])/F)=C_{m_1}.C_{m_2}$, with $m_2$ dividing $p^{r-1}(p-1)$.
Observe that this shows that if $\E$ has a $K(\z_{p^r})$-rational point of exact order $p^r$
(and in particular if $\E$ has a $K$-rational point of exact order $p^r$), then   
$K(\E[p^r])=K(\z_{p^r},\sqrt[m_1]{a})$. We also give some remarks on the possible istances when
$\QQ(\E[p])=\QQ(\z_{p},\sqrt[p]{a})$, with $a\in \QQ(\z_p)$. 

\bigskip In the second part of the paper, we concentrate in the case when $K$ is a number field and $\E: y^2=x^3+bx+c$, with $b,c\in K$. For all $m\geq 3$ we produce an upper bound $B(m,b,c)$, depending on $m$, $b$ and $c$, to the height of the discriminant of the extension $K(\E[m])/K$.  As a consequence we give an explicit effective version to the hypotheses of the following local-global question, known as local-global divisibility problem.  

\par\bigskip\noindent  \begin{problem} \label{prob1} Let $P\in {\mathcal{E}}(K)$. Assume that for all but finitely many places $v\in K$, there exists $D_v\in {\mathcal{E}}(K_v)$ such that $P=mD_v$, where $K_v$ is the completion of $K$ at the place $v$. 
Is it possible to conclude that there exists $D\in {\mathcal{E}}(K)$ such that $P=mD$?
\end{problem}

\noindent This problem was originally stated by Dvornicich and Zannier in 2001, 
in the more general setting of a commutative algebraic group
defined over $K$  \cite{DZ}. Many papers have been written about this question 
since its formulation. A solution to the problem for all powers of prime numbers implies a solution for every $m$.
The question has been completely answered for elliptic curves over $\QQ$. In particular the
local-global divisibility holds for all $m$ but the ones divided by the powers $p^n$, with
$p\in \{2,3\}$ and $n\geq 2$ (see \cite{DZ}, \cite{PRV2}, \cite{Cre} for further details). Moreover
the question has been answered for all but finitely many primes $p$ in elliptic curves over a general number field $K$ \cite{PRV}.
Concerning the formulation of the problem in the setting of commutative algebraic groups defined over
number fields, a classical anwer for algebraic tori of dimension 1 is given
by  the Grunwald-Wang Theorem and there are some partial answers in the case of algebraic tori 
of dimension $d>1$ \cite{Ill}, in the
case of abelian varieties  \cite{GR2},  \cite{GR3} and even in the general case 
\cite{Pal_17}. In addition the problem has been connected
with a classical question posed by Cassels
on the divisibility of elements of the Tate-Shafarevic group
$\Sha(K,\E)$ in the Weil-Ch\^atelet group $H^1(\Gal(\overline{K}/K,\E(K))$, where
$\overline{K}$ denotes the algebraic closure of $K$  \cite{Cre2}, \cite{CS}. In particular
the mentioned results produced in \cite{PRV} and \cite{PRV2} assure an affirmative answer to
Cassels' question for all but finitely many primes $p$ over $K$ (for all $p\geq 5$ when $K=\QQ$).

\bigskip  

In this paper we deal with the hypotheses of Problem \ref{prob1}. They are
not minimal,  in fact it suffices to assume that the local divisibility holds for a \emph{finite} number of places instead of all but finitely many places (see also \cite{P_CM}).
 Let $G:=\Gal(K(\E[m])/K)$ with cardinality $|G|$.  We will denote by $v$ both a prime in $K$ and the associate place and
by $N_{K/\QQ}(v)$ its norm. 
By $h(\alpha)$ we will denote the logarithmic height of $\alpha\in \overline{K}$. 
We will prove that
the assumption of the validity of the local divisibility for all but finitely many places $v$ in the
statement of Problem \ref{prob1}, can be replaced by  the assumption of the validity of the local divisibility for all $v$ with $h(N_{K/\QQ}(v))\leq 12577\cdot B(m,b,c)$, but a finite number of them with density $\delta < \ddfrac{1}{|G|}$, where
$B(m,b,c)$ is an upper bound of the height of the discriminant of the extension $L/K$. If $K/\QQ$ is a Galois extension, the condition $h(N_{K/\QQ}(v))\leq 12577\cdot B(m,b,c)$
is equivalent to $h(v)\leq 12577 \ddfrac{B(m,b,c)}{[K:\QQ]}$ and if $K=\QQ$, this is nothing but
$\log v\leq  12577\cdot B(m,b,c)$
(see Section \ref{sec1}). In particular, if $K=\QQ$ and $m$
is not divisible by any  power $p^r$,
with $p\in \{2,3\}$ and $r\geq 2$, then the validity of the local divisibility for all nonarchimedean
$v$ such that $\log v\leq 12577\cdot B(m,b,c)$, but a finite number of them with density $\delta < \ddfrac{1}{|G|}$, 
implies the global one.  
Even for some integers $m$ divisible by powers $2^r$ or $3^r$, with $r\geq 2$, for which the local-global divisibility does not hold in general, there are still examples of elliptic curves defined over $\QQ$ for which this Hasse principle for divisibility holds 
as well (see \cite{DZ} or \cite{Pal} for examples for $m=4$).
So, regarding the bound $B(m,b,c)$,  more generally we have that, when $K=\QQ$, in all the cases when the Hasse principle for divisibility holds,  the validity of the local divisibility by all $v$ with $\log v\leq 12577 B(m,b,c)$, but a finite number of them with density $\delta < \ddfrac{1}{|G|}$, implies the global one. 
In the last section of the paper, for $m=4$, in a case when the Hasse principle for divisibility holds, we will produce an explicit example 
showing how to find points not satisfying those hypotheses that are locally divisible for infinitely many primes but not globally divisible.
We will call them pseudodivisible points.

\section{On the $p^r$-division field} \label{sec_5}

\par  We recall the following result given in \cite{BP} about the 
number fields $K(\E[m])$, for all odd integers $m\geq 5$. 

\begin{thm} \label{ordinates1} Let $m\geq 5$ be an odd integer and let $\E$ be an elliptic curve defined over a field $K$ with $char(K)\neq 2,3$.
Let $P_1=\{x_1,y_1\}$ and $P_2=\{x_2,y_2\}$ be two $m$-torsion points of $\E$, generating $\E[m]$.
Then  $K(\E[m])=K(x_1,\z_m,y_2)$.
\end{thm}

\noindent In particular, when $m=p^r$, with $p>3$ and $r\geq 1$  we have

$$K(\E[p^r])=K(x_1,\z_{p^r},y_2).$$

\noindent We are going to show that, for all $r\geq 2$, we can replace $\z_{p^r}$ with $\z_p$ in this generating system.

\begin{thm} \label{ordinates2} Let $\E$ be an elliptic curve defined over a field $K$ with $char(K)\neq 2,3$.
Let $p>3$ be a prime number and let $r\geq 1$. If $P_1=\{x_1,y_1\}$ and $P_2=\{x_2,y_2\}$ are two $m$-torsion points of $\E$, generating $\E[m]$, then

$$K(\E[p^r])=K(x_1,\z_p,y_2).$$ 
\end{thm}

\begin{proof}
We first prove that $K[p^r]=K(x_1,x_2,\z_p,y_2)$.
Let $\sigma\in \Gal(K[p^r]/K)$ fixing $x_1$,  $x_2$, $\z_p$ and $y_2$.
Since $\sigma$ fixes $x_1$, $x_2$ and $y_2$, then $\sigma$ has the form

$$\left(\begin{array}{cc} \pm 1 & 0 \\ 0 & 1 \end{array}\right).$$

\noindent We recall that $\sigma(\z_{p^r})=\z_{p^r}^{\det(\sigma)}$ and, since $\z_p=\z_{p^r}^{p^{r-1}}$,
consequently $\sigma(\z_{p})=\z_{p}^{\det(\sigma)}$.  If $\sigma$ fixes $\z_p$, then $\det(\sigma)\equiv 1 \pmod  p$. We deduce

$$\sigma=\left(\begin{array}{cc}  1 & 0 \\ 0 &  1 \end{array}\right)$$

\noindent and $K[p^r]=K(x_1,\z_p,x_2,y_2)$. 
Now let $\sigma\in \Gal(K[p^r]/K)$ fixing $x_1$, $\z_p$ and $y_2$. To prove  $K[p^r]=K(x_1,\z_p,y_2)$,
it suffices to show that $\sigma$ is the identity. If $\sigma$ fixes $x_1$, then $\sigma$ has the form

$$\left(\begin{array}{cc} \pm 1 & \alpha \\ 0 & \beta \end{array}\right).$$

\noindent As above,  if $\sigma$ fixes $\z_p$, then $\det(\sigma)\equiv 1 \pmod{p}$. We deduce

$$\sigma=\left(\begin{array}{cc} \pm 1 & \alpha \\ 0 & \pm 1 +kp \end{array}\right),$$

\noindent with $0\leq k \leq p^{r-1}-1$. Moreover $\sigma$ fixes $y_2$. Then the polynomial $x_2^3+bx_2+c-y_2^2=0$ has degree at most
3 over $K$ and $[K[p^r]:K(x_1,\z_p,y_2)]\leq 3$. This implies that either $\sigma^2$ fixes $x_2$
or $\sigma^3$ fixes $x_2$. If  $\sigma^2(x_2)=x_2$, then 

$$\sigma^2\equiv\left(\begin{array}{cc}  1 & 0 \\ 0 & 1 \end{array}\right) \pmod {p^r},$$

\noindent i.e. 

$$\left(\begin{array}{cc}  1 & \pm 2\alpha +\alpha k p\\ 0 & 1\pm 2kp + k^2p^2 \end{array}\right)\equiv\left(\begin{array}{cc} 1 & 0 \\ 0 & 1 \end{array}\right) \pmod{p^r} .$$

\noindent Thus

$$\left\{ \begin{array}{lll}  
 \pm2\alpha +\alpha k p&\equiv 0 & \pmod {p^r} \\
 1\pm 2kp + k^2p^2&\equiv 1 & \pmod{p^r} 
\end{array}\right.$$

$$\left\{ \begin{array}{lll}  
 \alpha (\pm 2+k p)&\equiv 0 &\pmod {p^r} \\
 kp (\pm 2 + kp)&\equiv 0 &\pmod {p^r} 
\end{array}\right.$$

\noindent Since $p\neq 2$, then $\pm 2+kp\not\equiv 0 \pmod {p^r} $, for every $0\leq k \leq p^{r-1}-1$. Therefore $\alpha\equiv 0 \pmod{p^r}$ and $kp \equiv 0 \pmod{p^r}$, implying $\sigma=\pm\Id$. Anyway $\sigma=-\Id$ is not possible, since $\sigma$ fixes $y_2$.
Therefore $\sigma=\Id$.

\bigskip\noindent If $\sigma^3(x_2)=x_2$, then

$$\sigma^3\equiv \left(\begin{array}{cc}  1 & 0 \\ 0 & 1 \end{array}\right) \pmod {p^r},$$

\noindent i.e.

\begin{equation} \label{matrix} \left(\begin{array}{cc} 
\pm 1 & 3\alpha \pm 3\alpha k p+\alpha k^2 p^2\\
 0 & \pm 1 +3 k p\pm 3k^2p^2+k^3p^3\end{array}\right)
\equiv 
\left(\begin{array}{cc} 1 & 0 \\ 
0 & 1 \end{array}\right) \pmod {p^r}.
\end{equation}

\noindent Immediately we deduce 

$$\sigma=\left(\begin{array}{cc}  1 & \alpha \\ 0 &  \pm 1 +kp \end{array}\right).$$

\noindent Since $\det(\sigma)=\pm 1 +kp \equiv 1  \pmod p$, then 

$$\sigma=\left(\begin{array}{cc}  1 & \alpha \\ 0 &  1 +kp \end{array}\right).$$

\noindent Moreover, from \eqref{matrix}, we have

$$\left\{ \begin{array}{ll}  
 3\alpha + 3\alpha k p+\alpha k^2 p^2 & \equiv 0 \pmod {p^r} \\
 1 +3 k p+ 3k^2p^2+k^3p^3 & \equiv 1 \pmod {p^r} 
\end{array}\right.$$

$$\left\{ \begin{array}{ll}  
 \alpha (3+3k p+k^2p^2) & \equiv 0  \pmod {p^r} \\
 kp (3+3k p+k^2p^2) & \equiv 0 \pmod {p^r}.
\end{array}\right.$$

\noindent Owing to $3+3k p+k^2p^2\not\equiv 0 \pmod {p^r}$, for all $0\leq k \leq p^{r-1}-1$ (recall that $p> 3$), then $\alpha\equiv 0 \pmod{p^r}$ and $kp \equiv 0 \pmod {p^r}$, implying
$\sigma=\Id$ and $K(\E[p^r])=K(x_1,\z_p,y_2).$
\end{proof}

\noindent We keep the notation introduced in Section \ref{sec0} and set $F:=K(x_1,y_1)$.  As
usual we denote by $\FF_p$ the field with $p$ elements. 

\begin{thm}\label{deg_y_2_2}
For all $p> 3$ and $r\geq 1$, 
\begin{description}

\item[1.] the degree $[K_{p^r}:K(x_1,\z_{p^r})]$ divides $2p^r$ and the Galois group $\Gal(K_{p^r}/K(x_1,\z_{p^r}))$ is cyclic, generated
by a power of $\eta=\left(\begin{array}{cc} -1 & \hspace{0.3cm} 1 \\  \hspace{0.3cm} 0 & -1 \end{array}\right)\,$.

\item[2.] The extension  $K(\E[p^r])/F$ is metacyclic. In particular $K(\E[p^r])/F=F(\z_{p^r},\sqrt[m_1]{a})$,
with $a\in F(\z_{p^r})$ and  $\Gal(K(\E[p^r])/F)=C_{m_1}. C_{m_2}$,
where $m_1$, $m_2$ are positive integers such that $m_1|p^r$ and $m_2|p^{r-1}(p-1)$.
The group  $C_{m_1}$ is generated by a power of $\omega=\left(\begin{array}{cc} 1 &  1 \\   0 & 1 \end{array}\right)\,$.

\item[3.] If $\E$ has a $K(p^r)$-rational point of order $p^r$, then $K(\E[p^r])=K(\z_{p^r},\sqrt[m_1]{a})$,
with $a\in K(\z_{p^r})$ and $m_1|p^r$, and the extension $K(\E[p^r])/K$ is metacyclic of order dividing
$p^{2r-1}(p-1)$.

\end{description}

\end{thm}

\begin{proof} { }

\begin{description}
\item[1.] By Theorem \ref{ordinates1}, we have $K_{p^r}=K(x_1,\z_{p^r},y_2)$. Let $\s\in \Gal(K(\E[p^r])/K)$ fixing $x_1$ and
$\z_{p^r}$. Then
$\s(P_1)=\pm P_1$ and $\det(\s)=1$, i. e. 
$$\s=\left(\begin{array}{cc} \pm 1 & \alpha \\   0 & \pm 1 \end{array}\right),$$

\noindent for some $0\leqslant \alpha \leqslant p^r-1$. The powers of $\eta$ are
\[ \eta^n = \left\{ \begin{array}{ll} \left(\begin{array}{cc} 1 & -n \\ 0 & \hspace{0.3cm} 1 \end{array}\right) & {\rm if}\ n\ {\rm is\ even}\\
\ & \ \\
\left(\begin{array}{cc} -1 &   \hspace{0.3cm} n \\  \hspace{0.3cm} 0 & -1 \end{array}\right) & {\rm if}\ n\ {\rm is\ odd} \end{array} \right. \]
and its order is  $2p^r$. Clearly every power of $\s$ is a power of $\eta$ too. So
the Galois group $\Gal(K_{p^r}/K(x_1,\z_{p^r}))$ is cyclic of order dividing $2p^r$ and it is generated by a power of $\eta$.

\item[2.]
By Theorem \ref{ordinates1}, we have $K_{p^r}=K(x_1,\z_{p^r},y_2)$. In particular
$K_{p^r}=K(x_1,y_1,\z_{p^r},y_2)$. Let $\s\in \Gal(K(\E[p^r]/K)$ fixing $x_1$, $y_1$ and
$\z_{p^r}$. Then $\s(P_1)= P_1$ and $\det(\s)= 1$, i. e. 
$$\s=\left(\begin{array}{cc}  1 & \alpha \\   0 &  1  \end{array}\right),$$

\noindent for some $0\leqslant \alpha \leqslant p^r-1$. 
Thus $\Gal(K(\E[p^r])/F(\z_{p^r})$ is cyclic of order $m_1$ dividing $p^r$ and it is generated by a power of
$\omega$. Therefore
$K(\E[p^r])/F=F(\z_{p^r},\sqrt[m_1]{a})$, with $a\in F(\z_{p^r})$. 
The extension  $F(\z_{p^r})/F$ is cyclic of order dividing $p^{r-1}(p-1)$ (recall that $p\neq 2$). Therefore
$\Gal(K(\E[p^r])/F)=C_{m_1}. C_{m_2}$,
with $m_2| p^{r-1}(p-1)$.

\item[3.] This is a direct consequence of \textbf{2.}
\end{description}
\end{proof}

\noindent 
Notice that in part \textbf{3.} of Theorem \ref{deg_y_2_2} we could have $\sqrt[m_1]{a}\in K(\z_{p^r})$
and $K(\E[p^r])=K(\z_{p^r})$. Of course an interesting case is when
when $ \QQ(\E[p])=\QQ(\z_p,\sqrt[p]{a})$. The istances when $\QQ(\E[p])=\QQ(\z_p)$ have been investigated
in \cite{MS}, \cite{Reb}, \cite{Pal} and \cite{LRGJ}. We are going to give some remarks on 
elliptic curves without complex multiplication (CM in the following) 
defined over $\QQ$ such that $\QQ(\E[p])\subsetneq \QQ(\E[p])=\QQ(\z_p,\sqrt[p]{a})$.

\begin{pro}  Let $\E$ be a non CM elliptic curve defined over $\QQ$.
If $p\in \{2,3,5,7\}$, then there exists an elliptic curve $\E/\QQ$ such that
$\QQ(\E[p])/\QQ=\QQ(\z_{p},\sqrt[p]{a})$, with $a\in \QQ(\z_{p^r})\setminus (\QQ(\z_p))^p$.
If $p\geq 17$, $p\neq 37$, then $\QQ(\E[p])\neq \QQ(\z_{p},\sqrt[p]{a})$, for all $a\in \QQ(\z_{p})$.
\end{pro}

\begin{proof}
If $\QQ(\E[p])=\QQ(\z_{p},\sqrt[p]{a})$, with $a\in \QQ(\z_{p})\setminus (\QQ(\z_p))^p$,  then $|\Gal(\QQ(\E[p])/\QQ)|=p(p-1)$. By the classification of the maximal subgroup of $\GL_2(\FF_p)$ 
(see \cite{Ser}, in particular Proposition 15), we have that $p$ divides $|\Gal(\QQ(\E[p])/\QQ)|$ if and only if either $\Gal(\QQ(\E[p])/\QQ)$
is contained in a Borel subgroup, but it not contained in a split Cartan subgroup or
$\Gal(\QQ(\E[p])/\QQ)$ contains $\SL_2(\FF_p)$. If $p>17$ and $p\neq 37$, then the image of the representation

$$\rho_{\E,p}: \Gal(\QQ(\E[p])/\QQ)\longrightarrow \GL_2(\FF_p)$$

\noindent is either  $\GL_2(\FF_p)$ or it is contained in the normalizer of a non-split Cartan subgroup
\cite{Ser81}, \cite{BP}. Then  $\QQ(\E[p])\neq \QQ(\z_{p},\sqrt[p]{a})$. 
Therefore $\QQ(\E[p])= \QQ(\z_{p},\sqrt[p]{a})$ implies $p\in \Omega=\{2,3,5,7,11,13,17,37\}$. This proves the second
part of the proposition. We are going to show explicit examples of curves $\E$ defined over $\QQ$ and without CM,
such that $\QQ(\E[p])= \QQ(\z_{p},\sqrt[p]{a})$, with $a\in \QQ(\z_{p})\setminus (\QQ(\z_p))^p$, for all $p\leq 7$.
If $p=2$, it suffices to take the curve $$\E: y^2=(x-\sqrt{\alpha})(x+\sqrt{\alpha})(x-\beta),$$

\noindent with $\alpha,\beta\in \QQ$ and $\alpha$ not a rational square. In this case
 $\QQ(\E[2])=\QQ(\sqrt{\alpha})$. For $p=3$
consider the family of elliptic curves ${\mathcal{F}}_{b,a_0}$ in \cite[Theorem 4.1, part 3.]{BP0}, with
$a_0=2k^2$, $k\in \QQ$. For every curve $\E_{b,2k^2}$ of

$${\mathcal{F}}_{b,2k^2}: y^2=x^3+bx+\ddfrac{16b^2-864k^2b-3888k^4}{576k^2},$$

\noindent with $b,k\in \QQ$, we have  $\QQ(\E_{b,2k^2}[3])=\QQ\left(\z_3,\sqrt[3]{\frac{k^2b}{3}+4k^4}\right)$. 
So it suffices to take ${\mathcal{F}}_{b,2k^2}$, with $k\in \QQ$ such that $\ddfrac{k^2b}{3}+4k^4$ is not a cube.
By part \textbf{3.} of Theorem \ref{deg_y_2_2}, if $\E$ has a $\QQ(\z_p)$-rational point of exact order $p$, then 
$\QQ(\E[p])=\QQ(\z_p,\sqrt[n]{a})$, with $a\in \QQ(\z_p)$.
Since by Mazur's Theorem  \cite{Maz}, there exist elliptic curves with a rational torsion point of exact order $5$ or $7$, then
for those curves we have $\QQ(\E[5])=\QQ(\z_5,\sqrt[5]{a})$, with $a\in \QQ(\z_5)$ and, respectively  $\QQ(\E[7])=\QQ(\z_7,\sqrt[7]{a})$, with $a\in \QQ(\z_7)$.  For such curves there are two possibilities: either $\QQ(\E[p])=\QQ(\z_p)$ or $\QQ(\z_{p})\subsetneq \QQ(\E[p])=\QQ(\z_p,\sqrt[p]{a})$, $p\in \{5,7\}$. 
In \cite{LRGJ} Gonz\'ales-Jim\'enes and Lozano-Robledo proved that if $\QQ(\E[p])=\QQ(\z_{p})$, then
$p\leq 5$ (see also \cite{MS}, \cite{Reb}). Therefore for every elliptic curve with a rational torsion point of order $7$, 
we have $\QQ(\E[7])=\QQ(\z_7,\sqrt[7]{a}),$ with $a\in \QQ(\z_7)\setminus \QQ(\z_7)^7$. Examples of elliptic curves in
short Weierstrass form with a rational torsion point of order 7 are

$$\E: y^2=x^3-3483x+121014,$$

$$\E: y^2=x^3-1323x+6395814.$$

\noindent For $p=5$, examples of elliptic curves with a rational point of order $5$ such that $\QQ(\z_5)\subsetneq\QQ(\E[5])= \QQ(\z_5,
\sqrt[5]{a})$, with $a\in \QQ(\z_5)$ are the curves

$$\E: y^2=x^3-432x+8208,$$

$$\E: y^2=x^3-27x+55350.$$

\end{proof}

\noindent
There are no known examples of non CM elliptic curve defined over $\QQ$ 
such that $\QQ(\E[p])=\QQ(\z_p,\sqrt[n]{a})$, for $p\in \{11,13,17,37\}$.
The curve 121a1 in Cremona label is a non CM elliptic curve defined over
$\QQ(\z_{11})$ with 
a torsion point of exact order 11 defined over $K=\QQ(\z_{11}+\z_{11}^{-1})$. For this curve we have
$\QQ(\E[11])=\QQ(\z_{11},\sqrt[11]{a})$, for some $a\in \QQ(\z_{11})$; anyway  121a1 is defined over
$\QQ(\z_{11})$ and not over $\QQ$, as stated above. The curve 121b1 in Cremona label is  defined over $\QQ$
and it has a torsion point of exact order 11 defined over $K=\QQ(\z_{11}+\z_{11}^{-1})$, but it has CM too.
In principle we can have non CM elliptic curves
defined over $\QQ$ with $\QQ(\E[p])=\QQ(\z_p,\sqrt[n]{a})$, for $p\in \{11,13,17,37\}$. 
In fact, in principle we can have a non CM elliptic curve with a point of order
11 defined on a subfield of $\QQ(\z_{11})$, since we can have a non CM elliptic curve with a point of order
11 defined on a field of degree 5 over  $\QQ$ \cite{quartic}.
Similarly a point of exact order 13 of a non CM elliptic curve can be defined on a number field of degree
3 or 4 or 6 or 12 over $\QQ$ \cite{quartic} and a point of exact order 17 (resp. 37) of a non CM elliptic curve can be defined on a number field of degree 8 or 16 (resp. 12 or 36 over $\QQ$) \cite{quartic}. 
Anyway, regarding the cases when $p\in \{13,17,37\}$,
all the known Galois representations of $\Gal(\QQ(\E[p])/\QQ)$ in $\GL_2(\FF_p)$, that are 
not surjective 
are subgroups of $\GL_2(\FF_p)$ of order $> (p-1)p$ (then in particular $\QQ(\E[p])\neq \QQ(\z_p,\sqrt[n]{a})$ in all those examples)
\cite{Sut}. Therefore we can conjecture that if  $\QQ(\E[p])=\QQ(\z_{p},\sqrt[p]{a})$, with $a\in \QQ(\z_{p})\setminus (\QQ(\z_p))^p$, then $p\in S =\{2,3,5,7,11\}$. It is also our guess that $S$ can be shrunk to $\{2,3,5,7\}$.

\bigskip\noindent For completeness we bound the degree  $[K(\E[p^r]):K(x_1,\z_p)]$ and describe the Galois group
 $\Gal(K(\E[p^r])/K(x_1,\z_p))$. 

\begin{pro}\label{deg_y_2} Let $p\geq 5$ and $r\geq 1$. Then the degree $[K(\E[p^r]):K(x_1,\z_p)]$ divides $2p^{2r-1}$
and the extension $K(x_1,\z_p,y_2)/K(x_1,\z_p)$ is a metacyclic
extension with Galois group $C_{m_3}. C_{m_4}$,
where $m_3$, $m_4$ are positive integers such that $m_3|2p^r$ and $m_4|p^{r-1}$.
\end{pro}

\begin{proof}  
In part \textbf{1.} of Theorem \ref{deg_y_2_2} we have proved that $[K(\E[{p^r}])/K(x_1,\z_{p^r})]$ is a cyclic extension of order dividing $2p^r$. The extension $K(\z_{p^r})/K(\z_p)$
is cyclic  of order dividing $p^{r-1}$ (recall $p\neq 2$).  Since $K(\E[p^r])=K(x_1,\z_p,y_2)$ by Theorem
\ref{ordinates2}, then $[K(\E[p^r]):K(x_1,\z_p)]|2p^{2r-1}$
and $\Gal(K(\E[{p^r}])/K(x_1,\z_p))=C_{m_3}.C_{m_4}$, with $m_3| 2p^r$ and $m_4|p^{r-1}$.
\end{proof}

\section{On the height of the abscissas of the $m$-torsion points of $\E$} \label{sec3}

From now on let $K$ be a number field and $m\geq 3$. Let $M_K$ be the set of places $v$ of $K$.
For every $v\in M_K$, we denote by $|\hspace{0.1cm} . \hspace{0.1cm}|_v$ the associate absolute value. As in Section \ref{sec0},
by $K_v$ we denote the completion of $K$ at $v$.  
We briefly recall some basic facts about the height of a rational number
and of a polynomial (for further details see \cite{HS} and \cite{BG}).
We also show how to get a bound for the height of the abscissas of a
$m$-torsion point of $\E$, for every $m$.

\begin{D}
Let $\alpha\in K$.  We define the height of $\alpha$ as

$$H(\alpha)=\prod_{v\in M_K} \max\{1,|\alpha|_v\}^{\frac{d_v}{d}},$$

\noindent where $d=[K:\QQ]$ and $d_v=[K_v:\QQ_{\tilde{v}}]$, with $\tilde{v}\in M_{\QQ}$ and $v|\tilde{v}$.
\end{D}

\noindent  Observe that $H(v)=v^{\frac{d_v}{d}}$, for every prime $v$ of $K$. In particular
$H(v)=v$, when $K=\QQ$.
In many cases it is more useful to work with the logarithm of $H(\alpha)$, instead of $H(\alpha)$ itself. For this reason we have the following definition.

\begin{D}
Let $K$ be a number field and let $\alpha\in K$. We define the logarithmic height (or Weil height) of $\alpha$ as

$$h(\alpha):=\log^+ H(\alpha)=\sum_{v\in M_k}\frac{d_v}{d}\log^+|\alpha|_v.$$
\end{D}

\noindent We recall that $\log^+ 0:=0$ and $\log^+\beta:=\max\{0,\log \beta\}$ for every $\beta\in \RR^+$ 
(see for instance \cite{BG}). In particular we have $ h(\alpha)\geq 0$, for all $\alpha$. Let $\infty$
denote the archimedean place of $\QQ$ and $|\hspace{0.1cm} . \hspace{0.1cm}|$ denote the classical absolute value associated to
$\infty$.
If $\alpha\in \QQ$, $\alpha=\ddfrac{\alpha_1}{\alpha_2}$, with $\alpha_1\in \ZZ$, $\alpha_2\in \ZZ\setminus \{0\}$
and $\gcd(\alpha_1,\alpha_2)=1$, then
$h(\alpha)=\log^+\max\{|\alpha_1|, |\alpha_2|\}$. If $\alpha \in\ZZ$, then $h(\alpha)=\log^+|\alpha|$
and if $\alpha \in\ZZ\setminus \{0\}$, then $h(\alpha)=\log |\alpha|$.
There are some basic properties of the logarithmic height that 
we list below (see also \cite{Lan} and \cite{BG}).

\begin{pro} \label{height}
Let $\overline{\QQ}$ be the algebraic closure of $\QQ$ and let $\sigma\in \Gal(\overline{\QQ}/\QQ)$. Let $r$ be a positive integer
and let $\alpha,\beta,\alpha_1, \alpha_2, ..., \alpha_r\in \overline{\QQ}$. Then

\begin{description}

\item[i.] $h(\alpha\beta)\leq h(\alpha)+h(\beta);$ 

\item[ii.] $h(\alpha_1+\alpha_2+...+ \alpha_r)\leq h(\alpha_1)+h(\alpha_2)+...+h(\alpha_r)+\log r;$ 

\item[iii.] $h(\alpha^r)=r h(\alpha);$

\item[iv.] $h(\sigma(\alpha))=h(\alpha)$.

\end{description}
\end{pro}

There is a classical way of defining the height of a polynomial too.

\begin{D}
Let $K$ be a number field and let $f(x)\in K[x]$, such that 

$$f(x)=\sum_{i=0}^r a_ix^i.$$

\noindent We define the height of $f$ as

$$H(f)=\prod_{v\in M_k} (\max_i |a_i|_v)^{\frac{d_v}{d}} $$ 

\noindent and the logarithmic height of $f$ as

$$h(f)=\log^+H(f)=\sum_{v\in M_k} \frac{d_v}{d}  \log^+ (\max_i |a_i|_v).$$ 

\end{D}

\noindent For every algebraic integer $\alpha\in K$, we denote by $f_{\alpha}$ its minimal
polynomial and by $\widetilde{f_{\alpha}}$ the multiple of $f_{\alpha}$ with
integer coprime coefficients.  Observe that for $\widetilde{f_{\alpha}}$, we have $\max_i |a_i|_v=1$,
for every non-archimedean place $v$ of $K$; otherwise
all the coefficients of $\widetilde{f_{\alpha}}$ would be divisible by $v$, which is a contraddiction. 
Then

\begin{equation} \label{min_poly} h(\widetilde{f_{\alpha}})=\log^+ (\max_i |a_i|)= \log (\max_i |a_i|). \end{equation}

\noindent We recall the  following relation between the 
logarithmic height of $\alpha$ and the logarithmic height of $f_{\alpha}$.

\begin{pro} \label{height_min}  Let $K$ be a number field and $\alpha\in \overline{K}$ be algebraic over $K$,
with minimal polynomial $f_{\alpha}$. Then

\begin{equation} \label{falpha} h(\alpha) \leq h(f_{\alpha}) + \log 2. \end{equation} 
\end{pro}

\begin{proof}
Let $f_{\alpha}=\sum_{i=1}^n a_ix^i$, with $a_n=1$.
We prove that $|\alpha|_v\leq 2 \max_i |a_i|_v$, for every $v\in M_K$.  Obviously $\max_i |a_i|_v \geq 1$, because of $a_n=1$.
Therefore, if $|\alpha|_v \leq 1$, then $ |\alpha|_v \leq  \max_i |a_i|_v\leq 2 \max_i |a_i|_v$.
Assume $|\alpha|_v > 1$. Since $\alpha$ is a root of $f_{\alpha}$, then
$\alpha^n=-\sum_{i=1}^{n-1} a_i\alpha^i$. If $v$ is nonarchimedean, then, by the strong triangle
inequality, 

$$|\alpha|_v^n\leq  \max_i |a_i|_v |\alpha|_v^{n-1}$$ 

\noindent (recall that
we are assuming $|\alpha|_v > 1$, which implies  $|\alpha|_v^{n-1}>  |\alpha|_v^{i}$,
for every $1\leq i\leq n-2$).  Thus $|\alpha|_v\leq  \max_i |a_i|_v$. If $v=\infty$ is archimedean, then
we use the triangle inequality

\[\begin{split}
|\alpha|^n\leq & \sum_{i=1}^{n-1} |a_i||\alpha|^i\leq \max_i |a_i|\sum_{i=1}^{n-1}|\alpha|^i
=\max_i |a_i|\cdot |\alpha|^{n-1}\left(1+\sum_{i=1}^{n-2}\frac{1}{|\alpha|^i}\right)\\
& \leq \max_i |a_i|\cdot |\alpha|^{n-1}\left(1+\sum_{i=1}^{n-2}\frac{1}{2^i}\right) \leq 2 \max_i |a_i|\cdot |\alpha|^{n-1}
\end{split} \]

\noindent Therefore $|\alpha|\leq 2 \max_i |a_i|$. We deduce

$$\log^+ |\alpha|_v \leq \log^+|2|_v + \log^+ \max_i |a_i|_v,$$

\noindent for all $v\in M_K$, and then

$$\sum_{v\in M_k} \frac{d_v}{d} \log^+ |\alpha|_v \leq \sum_{v\in M_k} \frac{d_v}{d}\log^+|2|_v + \sum_{v\in M_k} \frac{d_v}{d}\log^+ \max_i |a_i|_v,$$

\noindent i.e. $$h(\alpha)= h(f_{\alpha}) +\log 2.$$
\end{proof}

 In the case when $\alpha$ is one of the abscissas of the $m$-torsion points of $\E$, the minimal polynomial
of $\alpha$ divides the $m$-th division polynomial $\Psi_m$ of $\E$ (i.e. the polynomial whose roots are the abscissas of the $m$-torsion
points of $\E$). We recall that 

$$\deg \Psi_m=\left\{\begin{array}{l} 
 \dfrac{m^2-1}{2}, \textrm{ if } m \textrm{ is odd};\\
\\
 \dfrac{m^2-4}{2}, \textrm{ if } m \textrm{ is even};\\
\end{array}\right.$$

\noindent (see for instance \cite{McK}) and that  $\Psi_m$ has integer coefficients. We denote by $b_i$, with $0\leq i\leq \deg \Psi_m$,
the coefficients of $\Psi_m$. We can deduce an uniform bound for the absolute value $|b_i|$,
for every $0\leq i\leq \deg \Psi_m$, by the following result  (see \cite[equation (6), pag. 769]{McK}). 

\begin{pro}[McKee, 2010] \label{mckee_prop}
Let $\Psi_m=\sum_{i=0}^{\deg \Psi_m} b_i x^i$ be the $m$-th division polynomial of $\E$, where
 $b_i=\sum_{2r+3s=\deg \Psi_m-i} a_{r,s}b^r c^s$. Then 

 \begin{equation} \label{bound} |a_{r,s}| \leq \frac{m^{m^2}(m^2-\frac{1}{2})!}{[(\frac{m^2-1}{2})!]^2(\frac{m^2}{2}+1)!} \sim \frac{2^{\frac{3m^2+1}{2}}e^{\frac{m^2}{2}}}{\pi m^3}, \end{equation} 

\noindent for all $ r, s$.
\end{pro}

\noindent  Observe that the bound in \eqref{bound} does not depend on ${r,s}$. So, in particular, it holds
for $\max |a_{r,s}|$. For $m$ small we have $$\frac{m^{m^2}(m^2-\frac{1}{2})!}{[(\frac{m^2-1}{2})!]^2(\frac{m^2}{2}+1)!} \leq \frac{2^{\frac{3m^2+1}{2}}e^{\frac{m^2}{2}}}{\pi m^3};$$

\noindent when $m$ grows, the bound of $|a_{r,s}|$ is asintothically equivalent to $\ddfrac{2^{\frac{3m^2+1}{2}}e^{\frac{m^2}{2}}}{\pi m^3}$. Therefore

 \begin{equation}  \log^+ |a_{r,s}| \leq \frac{3m^2+1}{2}\log 2 +\frac{m^2}{2} - 3\log m- \log \pi, \end{equation} 

\noindent for every $r,s$ and, in particular,

$$ \log^+\max  |a_{r,s}| \leq \frac{3m^2+1}{2}\log 2 +\frac{m^2}{2} - 3\log m- \log \pi. $$ 

\noindent By $b_i=\sum_{2r+3s=\deg \Psi_m-i} a_{r,s}b^r c^s$, one can deduce the following statement.

\begin{cor}
Let $\Psi_m=\sum_{i=0}^{\deg \Psi_m} b_i x^i$ be the $m$-th division polynomial of $\E$, where
$$\deg \Psi_m=\left\{\begin{array}{l} 
 \dfrac{m^2-1}{2}, \textrm{ if } m \textrm{ is odd};\\
\\
 \dfrac{m^2-4}{2}, \textrm{ if } m \textrm{ is even}.\\
\end{array}\right.$$

\noindent Then 
 \begin{equation} \label{bound2} \log^+\max_i  |b_i| \leq \deg \Psi_m\left(\dfrac{3m^2+1}{2}\log 2 +\dfrac{m^2}{2} - 3\log m- \log \pi +h(b) +h(c)\right).\end{equation}
\end{cor}

\noindent Notice that the last bound holds for the coefficients of $\widetilde{f_{\alpha}}$ too (and for the coefficients of $f_{\alpha}$ itself). With the next statement we give a bound to the height of the abscissas of a $m$-torsion point of $\E$, that
we will use in the following.

\begin{lem} Let $\E$ be an elliptic curve defined over a number field $K$, with Weierstrass form $\E: y^2=x^3+bx+c$. Let $\alpha$ be the abscissas of a $m$-torsion of $\E$ and let
$h(\alpha)$ denote its logarithmic height. Then

\begin{equation} \label{last} h(\alpha) \leq \left\{ \begin{array}{ll}  
(m^2-1)^2\log m +\dfrac{m^2-1}{2}\left(h(b) +h(c)\right)  +\log 2, & \textrm{if } m\geq 3  \textrm{ is odd; } \\
& \\
(m^2-4)^2\log m +\dfrac{m^2-4}{2}\left(h(b) +h(c)\right)  +\log 2,  & \textrm{if } m\geq 4  \textrm{ is even. } 
\end{array}
\right. 
\end{equation}

\end{lem}

\begin{proof}
By equation \eqref{min_poly} and equation \eqref{bound2} the height of $\widetilde{f_{\alpha}}$ can
be bounded as follows:

\begin{equation} \label{widetildef}  h(\widetilde{f_{\alpha}})\leq \deg \Psi_m \left(\dfrac{3m^2+1}{2}\log 2 +\dfrac{m^2}{2} - 3\log m- \log \pi +h(b) +h(c)\right). \end{equation}

\noindent By the Gelfand's inequality \cite[B.7.3]{HS}, we have that 
$$h(f_{\alpha})\leq h(\widetilde{f_{\alpha}})+ \deg(\widetilde{f_{\alpha}}).$$

\noindent Since $\deg(\widetilde{f_{\alpha}})\leq \deg(\Psi_m)$, then
by equation \eqref{widetildef} and equation  \eqref{falpha}, we deduce

 \begin{equation} \label{abs} h(\alpha) \leq 
\left\{ \begin{array}{ll}  
 \dfrac{m^2-1}{2}\left(\dfrac{3m^2+1}{2}\log 2  +\dfrac{m^2}{2} - 3\log m- \log \pi +h(b) +h(c)\right) 
& + \dfrac{m^2-1}{2} +\log 2,  \\
  & \textrm{if } m\geq 3  \textrm{ is odd; } \\
 & \\
&\\
\dfrac{m^2-4}{2}\left(\dfrac{3m^2+1}{2}\log 2 +\dfrac{m^2}{2} - 3\log m- \log \pi +h(b) +h(c)\right)
&+ \dfrac{m^2-4}{2} +\log 2,  \\
  & \textrm{if } m\geq 4  \textrm{ is even. }  \\
\end{array}
\right. \end{equation}

\noindent Suppose that $m\geq 3$ is odd. Observe that $\dfrac{3m^2+1}{2}\log 2 <\dfrac{3m^2+1}{2} \log m$ and
$\dfrac{m^2}{2}- \log \pi <\dfrac{m^2-1}{2}\log m$. Applying these inequalities, equation \eqref{abs} becomes

\[\begin{split}
h(\alpha) \leq &  \dfrac{m^2-1}{2}\left(\dfrac{3m^2+1}{2}\log m +\dfrac{m^2-1}{2}\log m - 3\log m +h(b) +h(c)\right) + \frac{m^2-1}{2} +\log 2\\
= & \dfrac{m^2-1}{2}\left(\dfrac{3m^2+1+m^2-1-6}{2}\log m +h(b) +h(c)\right) + \frac{m^2-1}{2} +\log 2\\
= &\dfrac{m^2-1}{2}\left((2m^2-3)\log m +h(b) +h(c)\right) + \frac{m^2-1}{2} +\log 2\\
\end{split}
\]

\noindent To give a bound in a more condensed form,  we also observe that
$\dfrac{m^2-1}{2}\leq \dfrac{m^2-1}{2}\log m$ (recall that $m\geq 3$).
Thus for $m\geq 3$ odd, we have

$$h(\alpha) \leq 
(m^2-1)^2\log m +\dfrac{m^2-1}{2}\left(h(b) +h(c)\right)  +\log 2.  $$

\noindent Now suppose that $m\geq 4$ is even. With the aim to get a bound similar to the one obtained for odd $m$, we observe
that  $$\dfrac{3m^2+1}{2}\log 2 +\dfrac{m^2}{2} - 3\log m- \log \pi \leq (2m^2-9)\log m.$$
\noindent Then equation \eqref{abs} becomes

\[\begin{split}
h(\alpha) \leq  \dfrac{m^2-4}{2}\left((2m^2-9)\log m +h(b) +h(c)\right) + \frac{m^2-4}{2} +\log 2.
\end{split}
\]

\noindent Again, to have a bound in a more condensed form notice that
$\dfrac{m^2-4}{2}\leq \dfrac{m^2-4}{2}\log m$, for all $m\geq 4$.
Thus

$$h(\alpha) \leq 
(m^2-4)^2\log m +\dfrac{m^2-4}{2}\left(h(b) +h(c)\right)  +\log 2. $$

\end{proof}
\normalcolor

\section{A bound for the height of the discriminant of a $m$-division field}

\bigskip\noindent For every extension $L/K$ of number fields, we denote by $D_{L/K}$ its discriminant. We recall that given a tower of extensions $L\subseteq F\subseteq K$,
the discriminant of $L/K$ is equal to

\begin{equation} \label{disc} D_{L/K}=D_{F/K}^{[L:F]}N_{F/K}(D_{L/F}).
\end{equation}

\bigskip\noindent (see for instance  \cite{Mar} or \cite{BLSS}). If $L=EF$ is the compositum of two fields
linearly disjoint over $K$, then

\begin{equation} \label{disc2} D_{L/K}=D_{E/K}^{[F:K]}D_{F/K}^{[E:K]}
\end{equation} 

\bigskip\noindent (see for instance \cite{Neu}). 

\begin{rem} \label{rem} 
Obviously we can rewrite equation  \eqref{disc} as

$$  D_{L/K}=D_{F/K}^{[L:F]}\prod_{\sigma\in \Gal(F/K)}\sigma(D_{L/F}).$$

\noindent Consider the logarithmic height of $|D_{L/K}|$ in the last equation. By
the properties recalled in Proposition \ref{height}, we have

\begin{equation} \label{disc_final}  h(D_{L/K})\leq [L:F]h(D_{F/K})+[F:K]h(D_{L/F}).
\end{equation}

\noindent  
On the other hand, if we consider the logarithmic height of $D_{E/K}$ in equation  \eqref{disc2}, then
we get

\begin{equation} \label{lin_disj}  h(D_{E/K})\leq [F:K]h(D_{E/K})+[E:K]h(D_{F/K}). \end{equation}

\noindent Since $E$ and $ F$ are linearly disjoint over $K$, then both
$[L:F]=[EF:F]=[E:K]$ and $D_{L/F}=D_{E/K}$. Therefore \eqref{lin_disj}
 is nothing but  \eqref{disc_final} again

$$  h(D_{E/K})\leq [F:K]h(D_{L/F})+[L:F]h(D_{F/K}).$$

\noindent Then, when we calculate the height of the discriminant on an extension $EF/K$, we can assume without loss of generality that $E$ and $F$ are linearly disjoint over $K$.
\end{rem}

\bigskip\noindent We also recall that $disc(f_1(x))$ divides $disc(f_1(x)f_2(x))$,
for every $f_1(x), f_2(x)\in K[x]$. From here on out we will consider the extension $L/K$ with  
$L=K(\E[m])$. We are going to give a bound to the height of its discriminant. 
Such a computation can be obtained by using one of the generating systems of $K(\E[m])/K$ mentioned above. 
Anyway, it turns out that the orders of the bounds got by using different generating systems 
are similar and then there is no much improvement in changing the generating system in this case. 
For this reason we are going to use the classical
generating system $\{x_1,x_2,y_1,y_2\}$ which has the advantage of holding for all $m$. 

\begin{thm} \label{DL_1}
Let $\E: y^2=x^3+bx+c$ be an elliptic curve defined over a number field $K$. Consider the extension
$L:=K(\E[m])$, where $m\geq 3 $ is a positive integer. Then

\begin{equation} \label{all_m} h(D_{L/K})\leq \left\{ \begin{array}{ll}  
5(m^2-1)^3(m^2-3)(\log m+h(b)+h(c)), & \textrm{if } m\geq 3  \textrm{ is odd; } \\
& \\
5(m^2-4)^3(m^2-6)(\log m+h(b)+h(c)),  & \textrm{if } m\geq 4  \textrm{ is even. } 
\end{array}
\right. 
\end{equation}
\end{thm}

\begin{proof} As in the previous sections, let $P_1=(x_1,y_1)$ and $P_2=(x_2,y_2)$ be a generating
set for $\E[m]$. Then $K(\E[m])=K(x_1,x_2,y_1,y_2)$.  To extimate $|D_{L/K}|$, we use the equalities \eqref{disc} and \eqref{disc2}.
We denote by $F_1$ and $F_2$ the fields $K(x_1)$ and respectively $K(x_2)$.  By $F_3$ 
we denote the field $K(x_1,x_2, y_1)=F_1F_2(y_1)$. First of all we have to bound the discriminant of the extension $F_1/K$.
Since this extension is monogeneous, we have to bound the discriminant of the minimal polynomial of $x_1$, which divides the $m$-th division
polynomial $\Psi_m$ of $\E$. Let $Disc(\Psi_m)$ be the discriminant of $\Psi_m$ and
let  $\Delta=-16(4b^3+27c^2)$ be the discriminant of $\E$. In \cite{Sch}, where $\Psi_m$ is denoted by $B_n^*$,
Schmidt shows

$$Disc(\Psi_m)=\left\{ \begin{array}{ll}  
 (-1)^{\frac{m-1}{2}} m^{\frac{m^2-3}{2}}\Delta^{\frac{(m^2-1)(m^2-3)}{24}}, & \textrm{if } m  \textrm{ is odd; } \\
 (-1)^{\frac{m-2}{2}}m^{\frac{m^2}{2}}2^{2-m^2}\Delta^{\frac{m^2(m^2+2)}{24}},  & \textrm{if } m  \textrm{ is even. } 
\end{array}
\right. $$

\noindent (see also \cite{Sta}). Thus 

\begin{equation} \label{discF1} D_{F_1/K} \leq \left\{ \begin{array}{ll}  
 (-1)^{\frac{m-1}{2}} m^{\frac{m^2-3}{2}}\Delta^{\frac{(m^2-1)(m^2-3)}{24}}, & \textrm{if } m  \textrm{ is odd; } \\
 (-1)^{\frac{m-2}{2}}m^{\frac{m^2}{2}}2^{2-m^2}\Delta^{\frac{m^2(m^2+2)}{24}},  & \textrm{if } m  \textrm{ is even. } 
\end{array}
\right. \end{equation}

\noindent Similarly, we have 

\begin{equation} \label{discF2} D_{F_2/K} \leq \left\{ \begin{array}{ll}  
 (-1)^{\frac{m-1}{2}} m^{\frac{m^2-3}{2}}\Delta^{\frac{(m^2-1)(m^2-3)}{24}}, & \textrm{if } m  \textrm{ is odd; } \\
 (-1)^{\frac{m-2}{2}}m^{\frac{m^2}{2}}2^{2-m^2}\Delta^{\frac{m^2(m^2+2)}{24}},  & \textrm{if } m  \textrm{ is even. } 
\end{array}
\right. \end{equation}

\noindent Since we aim to bound the height of the discriminant of the extension $F_1F_2/K$, 
by Remark \ref{rem}, we can assume, without loss of generality that $F_1$ and $F_2$ are linearly disjoint over $K$
and that

$$D_{F_1F_2/K}=D_{F_1/K}^{[F_2:K]}D_{F_2/K}^{[F_1:K]}.$$

\noindent  Both $[F_1:K]$ and $[F_2:K]$ are less or equal than $\deg(\Psi_m)$. As in the statement of Proposition \ref{mckee_prop}, we have

$$\deg(\Psi_m)=\left\{\begin{array}{l} 
 \dfrac{m^2-1}{2} \textrm{ if } m \textrm{ is odd};\\
\\
 \dfrac{m^2-4}{2} \textrm{ if } m \textrm{ is even}.\\
\end{array}\right.$$

\noindent By \eqref{discF1} and \eqref{discF2}, we get

\[ D_{F_1F_2/K} \leq \left\{ \begin{array}{lll}  
\left[(-1)^{\frac{m-1}{2}} m^{\frac{m^2-3}{2}}\Delta^{\frac{(m^2-1)(m^2-3)}{24}}\right]^{m^2-1} & \hspace{0.1cm} \textrm{if } m   \textrm{ is odd; }\\ 
\left[(-1)^{\frac{m-2}{2}}m^{\frac{m^2}{2}}2^{2-m^2}\Delta^{\frac{m^2(m^2+2)}{24}}\right]^{m^2-4}  & \hspace{0.1cm}\textrm{if } m  \textrm{ is even. } \\
\end{array}
\right. 
\]

\noindent Thus

$$h(D_{F_1F_2/K}) \leq \left\{ \begin{array}{lll}  
\dfrac{m^2-1}{2} (h(D_{F_1/K})+h(D_{F_2/K})), & \textrm{if } m   \textrm{ is odd; }\\ 
\\
\dfrac{m^2-4}{2} (h(D_{F_1/K})+h(D_{F_2/K})),  & \textrm{if } m  \textrm{ is even; } \\
\end{array}
\right. $$

\noindent and  

\begin{equation} \label{disj}
h(D_{F_1F_2/K})\leq \left\{ \begin{array}{lll}  
 (m^2-1)\left({\dfrac{m^2-3}{2}}\log m+{\dfrac{(m^2-1)(m^2-3)}{24}}h(\Delta)\right), & \textrm{if } m  \textrm{ is odd; } \\
\\
 (m^2-4)\left({\dfrac{m^2}{2}}\log m+{(m^2-2)}\log 2 +{\dfrac{m^2(m^2+2)}{24}}h(\Delta)\right),  & \textrm{if } m  \textrm{ is even. } \\
\end{array}
\right.
\end{equation}

\noindent We must extimate the norm of $D_{F_3/F_1F_2}$ and the norm
of $D_{L/F_1F_2}$. Since $F_3$ contains $x_1$, then $[F_3:F_1F_2]\leq 2$. 
If $F_3\neq F_1F_2$, then a basis of $F_3/F_1F_2$ is ${\mathfrak{B}}=\{1,y_1\}$. The Galois group $\Gal(F_3/F_1F_2)$
is generated by the automorphism of the $m$-torsion points of $\E$ sending $y_1$ to $-y_1$
(and fixing $x_1$). Therefore

\begin{equation} \label{disc_y} Disc({\mathfrak{B}})=\left [\det \left( \begin{array}{cc}
1 & y_1 \\
1 & -y_1 \\ 
\end{array} \right) \right]^2=(2y_1)^2=4y_1^2=4(x_1^3+bx_1+c).
\end{equation} 

\noindent  Similarly the discriminant of a basis of the extension  $F_1F_2(y_2)$ is $4(x_2^3+bx_2+c)$.
Because of Remark \ref{rem} again, we can assume without loss of generality that $F_3$ and $F_1F_2(y_2)$ are linearly disjoint over $F_1F_2$. 
Then by the equality \eqref{disc2} we have $D_{L/F_1F_2}\leq 2^8(x_1^3+bx_1+c)^2(x_2^3+bx_2+c)^2$. By the equality \eqref{disc}, there is such a relation between the
discriminant of $L/K$ and the ones of $F_1F_2/K$ and of $L/F_1F_2$

$$D_{L/K}=D_{F_1F_2/K}^{[L:F_1F_2]}N_{F_1F_2/K}(D_{L/F_1F_2}).$$

\noindent We have

$$N_{F_1F_2/K}(D_{L/F_1F_2})=\prod_{\sigma\in \Gal(F_1F_2/K)} 2^8\sigma(x_1^3+bx_1+c)^2\sigma(x_2^3+bx_2+c)^2.$$

\noindent Since both $x_1$ and $x_2$ are roots of $\Psi_m$, then  $[F_1F_2:K] \leq \deg \Psi_m(\deg \Psi_m-1)$, i.e.

$$[F_1F_2:K] \leq  \left\{ \begin{array}{lll}  
\ddfrac{(m^2-1)(m^2-3)}{4}, & \textrm{if } m   \textrm{ is odd; }\\ 
\\
\ddfrac{(m^2-4)(m^2-6)}{4},  & \textrm{if } m  \textrm{ is even. } \\
\end{array}
\right. $$

\noindent For all odd $m\geq 3$, by Proposition \ref{height} we get

\[
\begin{array}{lll} h(N_{F_1F_2/K}(D_{L/F_1F_2}))   & \leq
[F_1F_2:K]  [  8\log 2+ 2h(x_1^3+bx_1+c)  +2h(x_2^3+bx_2+c)]\\
& \leq [F_1F_2:K]  [  8\log 2+ 2 ( 3 h(x_1) + h(x_1) +h(b)+h(c) +\log3 )  \\
& \hspace{3cm} + 2 (3 h(x_2) + h(x_2) +h(b)+h(c) +\log 3 ) ]. \\
\end{array}
\]

\noindent Then
\begin{equation} \label{both}
 h(N_{F_1F_2/K}(D_{L/F_1F_2}))    \leq \left\{\begin{array}{lll} 
   \ddfrac{(m^2-1)(m^2-3)}{4}[  8 \log 2+  8 h(x_1)  +8 h(x_2)&\\
   \hspace{3cm} +4h(b)+ 4h(c) +4\log 3], & \textrm{ if } m\geq 3 \textrm{ is odd};\\
  \ddfrac{(m^2-4)(m^2-6)}{4}[  8 \log 2+  8 h(x_1)  +8 h(x_2) &\\
   \hspace{3cm}  +4h(b)+ 4h(c)+4\log 3], & \textrm{ if } m\geq 3 \textrm{ is even}. \\
\end{array} \right.
\end{equation}

\noindent   We are going to treat separately the case when $m$ is odd and the case when $m$ is even. Firstly assume $m\geq 3 $ odd. 
By inequality \eqref{last}, for $i=1, 2,$ we get

\[
\begin{split} h(N_{F_1F_2/K}(D_{L/F_1F_2}) )  \leq & \ddfrac{(m^2-1)(m^2-3)}{4}\bigg[8\log 2+  16 \bigg( (m^2-1)^2\log m + \\
& \left.\frac{m^2-1}{2}(h(b) +h(c))  +\log 2\bigg) +4h(b) + 4h(c) +4\log 3\right] \\
 = & (m^2-1)(m^2-3)\bigg[2\log 2 +4 \bigg( (m^2-1)^2\log m +\\
&\frac{m^2-1}{2}\left(h(b) +h(c)\right)  +\log 2\bigg) +h(b)+h(c) +\log 3\bigg].\\ 
\end{split}
\]

\noindent In order to give a more elegant bound, we observe that $\log 2\leq \log m$ and $\log 3\leq \log m$, 
for every $m\geq 3$.  Then

\[
\begin{split} h(N_{F_1F_2/K}(D_{L/F_1F_2}) )  \leq &(m^2-1)(m^2-3)\bigg[2\log m +4 \bigg( (m^2-1)^2\log m \\
&+\frac{m^2-1}{2}\left(h(b) +h(c)\right)  +\log m\bigg) +h(b)+h(c) +\log m\bigg]\\ 
=& (m^2-1)(m^2-3)\bigg[(4(m^2-1)^2+7)\log m +(2m^2-1)\left(h(b) +h(c)\right) \bigg].\\ 
\end{split}
\]

\noindent In addition $4(m^2-1)^2+7\leq \dfrac{9}{2}(m^2-1)^2$ and $2m^2-1\leq \dfrac{9}{2}(m^2-1)^2$. Thus

\begin{equation} \label{eleg}
\begin{split}h(N_{F_1F_2/\QQ}(D_{L/F_1F_2})) & \leq  \dfrac{9}{2}(m^2-1)^3(m^2-3)(\log m +h(b) +h(c)). \\
\end{split}
\end{equation}

\noindent Putting together all those considerations and using equation \eqref{disc}, we deduce, for $m\geq 3$ odd,

\[ \begin{split}
h(D_{L/K}) \leq & 4\frac{(m^2-1)(m^2-3)}{2} \log m + 4 {\frac{(m^2-1)^2(m^2-3)}{24}}h(\Delta) \\
&+  \dfrac{9}{2}(m^2-1)^3(m^2-3)(\log m +h(b) +h(c))\\
 = & 2(m^2-1)(m^2-3) \log m +  {\frac{(m^2-1)^2(m^2-3)}{6}}h(\Delta) \\
&+  \dfrac{9}{2}(m^2-1)^3(m^2-3)(\log m +h(b) +h(c)).\\
\end{split} \]

\noindent Observe that 

\begin{equation} \label{delta} \begin{split}   h(\Delta) \leq & 6\log 2+3h(b)+3\log 3 +2h(c) +\log 2\\
& = 7\log 2+3\log 3 +  3h(b) +2h(c)\\
& \leq 10\log m+ 3h(b) +2h(c).\\
\end{split} \end{equation} 

\noindent Thus

\[\begin{split}
h(D_{L/K}) 
\leq &  2(m^2-1)(m^2-3) \log m + {\frac{(m^2-1)^2(m^2-3)}{6}}(10\log m+ 3h(b) +2h(c))\\
&+  \dfrac{9}{2}(m^2-1)^3(m^2-3)(\log m +h(b) +h(c))\\
=&   (m^2-1)(m^2-3) \left( \left(2+ \frac{5}{3}(m^2-1)\right)\log m + {\frac{m^2-1}{2}}h(b) +\frac{m^2-1}{3}h(c)\right)\\
&+  \dfrac{9}{2}(m^2-1)^3(m^2-3)(\log m +h(b) +h(c))\\
=&   (m^2-1)(m^2-3) \left( \frac{5m^2+5}{3}\log m + {\frac{m^2-1}{2}}h(b) +\frac{m^2-1}{3}h(c)\right)\\
&  +\dfrac{9}{2}(m^2-1)^3(m^2-3)(\log m +h(b) +h(c))\\
=&  (m^2-1)(m^2-3) \left[ \left(\frac{5m^2+5}{3}+\dfrac{9}{2}(m^2-1)^2\right)\log m \right.\\
&\left. +\left(\frac{m^2-1}{2}+\dfrac{9}{2}(m^2-1)^2\right)h(b)+
\left(\frac{m^2-1}{3}+\dfrac{9}{2}(m^2-1)^2\right)h(c).\right]
\end{split} \]

\noindent Observe that $\left(\dfrac{5m^2+5}{3}+\dfrac{9}{2}(m^2-1)^2\right)\leq 5(m^2-1)^2$,  for all $m\geq 3$. 
Moreover $\left(\dfrac{m^2-1}{2}+\dfrac{9}{2}(m^2-1)^2\right)\leq 5(m^2-1)^2$ and $\left(\dfrac{m^2-1}{3}+\dfrac{9}{2}(m^2-1)^2\right)\leq 5(m^2-1)^2$, for all $m\geq 3$. Then 

\[
\begin{split}
h(D_{L/K}) 
 & \leq  5(m^2-1)^3(m^2-3)(\log m+h(b)+h(c)).\\
\end{split}
\]

\noindent Assume that $m\geq 4$ is even. By equation \eqref{both} and equation \eqref{last}, we get

\[
\begin{split} h(N_{F_1F_2/K}(D_{L/F_1F_2}) )  \leq & \ddfrac{(m^2-4)(m^2-6)}{4}\left[8\log 2+  16 \left( (m^2-4)^2\log m +\right. \right.\\
& \frac{m^2-4}{2}(h(b) +h(c))  +\log 2) +4h(b) \left.+ 4h(c) +4\log 3\right] \\
=& 
(m^2-4)(m^2-6)\bigg[2\log 2 +4 \bigg( (m^2-4)^2\log m +\\
&\frac{m^2-4}{2}\left(h(b) +h(c)\right)  +\log 2\bigg) +h(b)+h(c) +\log 3\bigg].\\ 
\end{split}
\]

\noindent Again, in order to give a more elegant bound, 
we use that $\log 2\leq \log m$ and $\log 3\leq \log m$, 
for every $m\geq 4$. Then

\[
\begin{split} h(N_{F_1F_2/K}(D_{L/F_1F_2}) ) \leq &
(m^2-4)(m^2-6)\bigg[2\log m +4 \bigg( (m^2-4)^2\log m +\\
&\frac{m^2-4}{2}\left(h(b) +h(c)\right)  +\log m\bigg) +h(b)+h(c) +\log m\bigg].\\ 
= &  (m^2-4)(m^2-6)\bigg[(4(m^2-4)^2+7)\log m +(2m^2-7)\left(h(b) +h(c)\right) \bigg]\\ 
\end{split}
\]

\noindent We have $4(m^2-4)^2+7\leq \dfrac{9}{2}(m^2-4)^2$ and $2m^2-7\leq \dfrac{9}{2}(m^2-4)^2$. Thus

\begin{equation} \label{eleg_2}
\begin{split}h(N_{F_1F_2/\QQ}(D_{L/F_1F_2})) & \leq  \dfrac{9}{2}(m^2-4)^3(m^2-6)(\log m +h(b) +h(c)). \\
\end{split}
\end{equation}
We can repeat the same arguments as above,
by using the bound \eqref{disj} for $m$ even, i.e.

\begin{equation} \label{last_eq_2} 
h(D_{F_1F_2/K}) \leq  (m^2-4)\left[{\frac{m^2}{2}}\log m+{(m^2-2)}\log 2 +{\frac{m^2(m^2+2)}{24}}h(\Delta)\right].
\end{equation}

\noindent By equation \eqref{disc}, we get

\[
\begin{split}
h(D_{L/K})  \leq & 4(m^2-4)\left[{\frac{m^2}{2}}\log m+{(m^2-2)}\log 2 +{\frac{m^2(m^2+2)}{24}}h(\Delta)\right] \\
&+ \dfrac{9}{2}(m^2-4)^3(m^2-6)(\log m +h(b) +h(c)).\\
\end{split} 
\]

\noindent We use again $\log 2\leq \log m$. Then

\[\begin{split}
h(D_{L/K})  \leq & 4(m^2-4)\left[{\frac{m^2}{2}}\log m+ {(m^2-2)}\log m  +{\frac{m^2(m^2+2)}{24}}h(\Delta)\right] \\
&+ \dfrac{9}{2}(m^2-4)^3(m^2-6)(\log m +h(b) +h(c))\\
 = & 4(m^2-4)\left[\ddfrac{3m^2-4}{2}\log m  +{\frac{m^2(m^2+2)}{24}}h(\Delta)\right] \\
&+ \dfrac{9}{2}(m^2-4)^3(m^2-6)(\log m +h(b) +h(c))\\
=  & (m^2-4)\left[(6m^2-8)\log m  +{\frac{m^2(m^2+2)}{6}}h(\Delta)\right] \\
&+ \dfrac{9}{2}(m^2-4)^3(m^2-6)(\log m +h(b) +h(c)).\\
\end{split} \]

\noindent Because of equation \eqref{delta}, i.e. $h(\Delta)\leq 10\log m+3 h(b)+2h(c)$, we get

\[\begin{split}
h(D_{L/K}) \leq & (m^2-4)\left[(6m^2-8)\log m  +{\frac{(m^2+2)m^2}{6}}(10\log m+3 h(b)+2h(c))\right] \\
&+\dfrac{9}{2}(m^2-4)^3(m^2-6)(\log m +h(b) +h(c))\\
 = & \bigg[ \left(6m^2-8+\frac{5}{3}(m^2+2)m^2\right)(m^2-4)+\dfrac{9}{2}(m^2-4)^3(m^2-6)\bigg]\log m\\
& +\left( {\frac{m^2+2}{2}}(m^2-4)m^2+\dfrac{9}{2}(m^2-4)^3(m^2-6)\right)h(b)\\
& +\left( {\frac{m^2+2}{3}}(m^2-4)m^2+\dfrac{9}{2}(m^2-4)^3(m^2-6)\right)h(c).\\
\end{split} \]

\noindent We have $\left(6m^2-8+\dfrac{5}{3}(m^2+2)m^2\right)(m^2-4)+\dfrac{9}{2}(m^2-4)^3(m^2-6)\leq 5(m^2-4)^3(m^2-6)$, for all
 $m\geq 4$. In addition ${\dfrac{m^2+2}{2}}(m^2-1)m^2 +\dfrac{9}{2}(m^2-4)^3(m^2-6)\leq 5(m^2-4)^3(m^2-6)$ and ${\dfrac{m^2+2}{3}}(m^2-1)m^2+\dfrac{9}{2}(m^2-4)^3(m^2-6) \leq 5(m^2-4)^3(m-6)$, for all
$m\geq 4$. Then 

\[ \begin{split}
h(D_{L/K}) \leq  & 
 5(m^2-4)^3(m-6)(\log m+h(b) +h(c)).\\
\end{split}
\]

\end{proof}

\noindent    We will use the bounds \eqref{all_m}
 in the next section to give an explicit effective version of the hypotheses of Problem \ref{prob1},
in elliptic curves over $\QQ$.

\begin{rem} \label{rem_eleg} 
As recalled in Section \ref{sec3}, if $b=\ddfrac{b_1}{b_2}$, $c=\ddfrac{c_1}{c_2}$, with 
$b_1, c_1\in \ZZ$, $b_2, c_2\in \ZZ\setminus \{0\}$, $\gcd(b_1,b_2)=1$ and $\gcd(c_1,c_2)=1$, 
then $h(b)=\log^+\max\{|b_1|,|b_2|\}$ and $h(c)=\log^+\max\{|c_1|,|c_2|\}$.
If $b,c\in \ZZ$, then $h(b)=\log^+|b|$ and $h(c)=\log^+|c|$. In particular, if
 $b,c\in \ZZ\setminus\{0\}$, then 
$h(b)+h(c)= \log |bc|$. 
In this last case the bound \eqref{all_m} can also be expressed in the following more elegant form

$$|h(D_{L/K})|\leq  \left\{ \begin{array}{ll}  
(m^2-1)^3(m-3)\log m^{5}|bc|, & \textrm{if } m\geq 3  \textrm{ is odd; } \\
& \\
(m^2-4)^3(m-6)\log m^{5}|bc|,  & \textrm{if } m\geq 4  \textrm{ is even. } 
\end{array}
\right. $$

\end{rem}

\section{An explicit effective version of the hypotheses of Problem \ref{prob1}} \label{sec1}

\bigskip\noindent We briefly recall the definition of the cohomology group which gives an obstruction
to the validity of the Hasse principle for divisibility of points in $\E/K$.
Let $P\in \E(K)$ and let $D\in \E(\overline{K})$ be a $m$-divisor of
$P$, i.e. $P=mD$. For every $\sigma\in G={\Gal}(K(\E[m])/K)$, we have

$$m\sigma(D)=\sigma(mD)=\sigma(P)=P.$$

\noindent Thus $\sigma(D)$ and $D$ differ by a point in $\E[m]$ 
and one can define a cocycle $\{Z_{\sigma}\}_{\sigma\in G}$ of $G$ with
values in $\E[m]$ by

\begin{equation} \label{eq1} Z_{\sigma}:=\sigma(D)-D. \end{equation} 

\noindent Such a cocycle vanishes in $H^1(G,\E[m])$, if and only if there exists a $K$-rational
$m$-divisor of $P$  \cite{DZ}. In particular, the hypotheses
about the validity of the local-divisibility in Problem \ref{prob1}, assuring the existence of a $K_v$-rational
$m$-divisor  of $P$, imply that
the cocycle $\{ Z_{\sigma}\}_{\sigma\in G}$ vanishes in  $H^1({\Gal}((K(\E[m]))_w/K_v),\E[m])$,
for all but finitely many $v\in M_K$, where $w$ denotes a place of $K(\E[m])$ extending $v$. 
This fact motivates the following definition \eqref{h1loc}, given by Dvornicich and Zannier in \cite{DZ}, of a subgroup 
of $H^1(G,\E[m])$ which encodes the hypotheses of the
problem in this cohomological context and 
gives an obstruction to the validity of this Hasse principle  \cite{DZ3}.
Let $G_v$ denote the group ${\Gal}((K(\E[m]))_w/K_v)$
and let $\Sigma$ be the subset of $M_K$ containing all the $v\in M_K$, that are
unramified in $K(\E[m])$. Then

\begin{equation} \label{h1loc}
H^1_{\textrm{\loc}}(G,\E[m]):=\bigcap_{v\in \Sigma} (\ker  H^1(G,\E[m])\xrightarrow{\makebox[1cm]{{\small $res_v$}}} H^1(G_v,\E[m])),
\end{equation}

\noindent where $res_v$ is the usual restriction map. 

\noindent Since every $v\in \Sigma$ is unramified in $K(\E[m])$,  then $G_v$ is a cyclic subgroup of $G$, for all $v\in \Sigma$.
By the Chebotarev Density Theorem, the local Galois group $G_v$ varies over \emph{all} cyclic subgroups of
$G$ as $v$ varies in $\Sigma$ (see Theorem \ref{Cheb} and Theorem \ref{Lan} below for further details). 
Then we have the following equivalent definition of the group $H^1_{\textrm{\loc}}(G,\E[m])$.

\begin{D} 
 A cocycle $\{Z_{\sigma}\}_{\sigma\in G}\in H^1(G,\E[m])$ satisfies the
\emph{local conditions} if, for every $\sigma\in G$, there exists $A_{\sigma}\in \E[m]$ such that
$Z_{\sigma}=(\sigma-1)A_{\sigma}$. The subgroup of $H^1(G,\E[m])$ formed by all the cocycles satisfying the local conditions
is the \emph{first local cohomology group} $H^1_{\loc}(G,\E[m])$.
\end{D}

\noindent The triviality of $H^1_{\textrm{\loc}}(G,\E[m])$ assures the validity of the local-global divisibility by $m$ 
in $\E$ over $K$ \cite[Proposition 2.1]{DZ}. The condition $H^1_{\loc}(G,\E[m])=0$ is also necessary, not exactly over $K$, but over
a finite extension of $K$  \cite[Theorem 3]{DZ3}. So the nontriviality of the first cohomology
group $H^1_{\loc}(G,\E[m])$ is an obstruction to the validity of the Hasse principle.
\par As stated above, the validity of the Hasse principle for divisibility of points
in $\E$ has been proved for many integers $m$. Anyway in all the other papers (of various authors) about this topic, there
is no information about the minimal number of places $v$ for which the validity of the local
divisibility 
over $K_v$ is sufficient to have the global divisibility 
in $\E$ over $K$, for a general $m$.
Only when $m=5$ and $\E$ is an elliptic curve with Weierstrass form $y^2=x^3+bx$ or $y^2=x^3+c$, with $b,c\in \QQ$, minimal bounds were produced in \cite{P_CM}. 
 For the first time, here we show for every $m\geq 3$, an explicit upper bound  to the number of places $v$ for which the validity of the local
divisibility by $m$ implies the global one in all the cases when the Hasse principle for divisibility holds, in elliptic curves
defined over $\QQ$. As recalled in Section \ref{sec0}, this in particular happens in all $\E$ over $\QQ$, for every $m$ not divisible by any power $p^n$, with $p\in \{2,3\}$ and $n\geq 2$, but there are also examples of elliptic curves $\E$ over $\QQ$
for which the local-global principle for divisibility holds when $m$ is divisible by powers of 2 or 3 (with $n\geq 2$).

  We have already
mentioned that by the Chebotarev Density Theorem, the group $G_v$ varies over all
the cyclic subgroups of $G$, as $v$ varies among all the places of $K$, that are unramified in $K(\E[m])$.
Therefore in fact we have

$$H^1_{\loc}(G,\E[m])=\bigcap_{v\in S} (\ker  H^1(G,\E[m])\xrightarrow{\makebox[1cm]{{\small $res_v$}}} H^1(G_v,\E[m])),$$

\noindent where $S$ is a subset of $\Sigma$ such that $G_v$ varies
over all cyclic subgroups of $G$ as $v$ varies in $S$.
If we are able to find such a set $S$, then we can replace the hypotheses of Problem \ref{prob1} about the
validity of the local divisibility for all but finitely many $v\in M_K$ with the assumption of the validity of the local 
divisibility for every $v\in S$. Observe that in particular $S$ is finite (on the contrary $\Sigma$ is not finite). So it suffices to have that
the local divisibility by $m$ holds for a finite number of suitable places to get the global divisibility by $m$.
Of course $S$ varies with respect to $m$ and $\E$. In \cite{P_CM} the minimal possible cardinality of the set $S$ is showed when
$m=5$ and $\E$ is an elliptic curve with Weierstrass form $y^2=x^3+bx$ or $y^2=x^3+c$, with $b,c\in \QQ$.   
In particular the maximum number $N$ of cyclic subgroups of $\Gal(K(\E[5])/K)$ is calculated. 
Then $N$ is used as a lower bound for the cardinality of $S$. In principle, for every $m$, we can have a set $S$, as above, with cardinality $N$ equal to the maximum number of cyclic subgroups of $\Gal(K(\E[m])/K)$. Anyway
it is not immediate to define a proper set $S$, by using only this information about its cardinality. 
In fact, we are not sure that the validity of the local
divisibility for the first $N$ rational primes, implies the validity of the global divisibility. 
A priori two different primes among the first $N$ ones, can correspond to the same
cyclic Galois group $G_v$.  
For a general $\E: y=x^3+bx+c$ and for every $m\geq 3$, we are going to show that we can take

\begin{equation} \label{S} S=\{v\in M_K\setminus\{\infty\}| h(N_{K/\QQ}(v))\leq 12577 \cdot B(m,b,c)\}\setminus S', \end{equation}

\noindent where $B(m,b,c)$ is the bound for $h(D_{L/K})$ appearing in Theorem \ref{DL_1} (respectively for $m\geq 3$ odd and
$m\geq 4$ even) and $S'$ is
a subset of 
$$\{v\in M_K\setminus\{\infty\}| h(N_{K/\QQ}(v))\leq 12577 \cdot B(m,b,c)\}$$ 

\noindent with cardinality $|S'| < \ddfrac{1}{|G|}.$ In particular, we are going to see that for every cyclic subgroup $C$ of $G$, there exists a place $v$ with $h(N_{K/\QQ}(v))\leq 12577\cdot B(m,b,c)$, such that $G_v=C$.

\begin{thm}[Chebotarev Density Theorem, 1926] \label{Cheb} Let $L/K$ be a finite Galois extension.
For every prime $v$ of $K$, unramified in $L$, let $\left(\ddfrac{L|K}{v}\right)$
denote the Artin symbol of $v$. For every conjugacy class $C$ in $\Gal(L/K)$
the density of the primes $v$ such that  $\left(\ddfrac{L|K}{v}\right)=C$
is $\ddfrac{\#C}{\#\Gal(L/K)}$. 
\end{thm}

\bigskip\noindent  Since the smallest among the cardinalities of a coniugacy class $C$ in $G=\Gal(L/K)$ is 1, 
 by the Chebotarev Density Theorem,  the density $\delta(C)$ of primes $v$ such that $G_v={C}$
is boundend in the following way

\begin{equation} \label{density}
\ddfrac{1}{|G|}\leq \delta(C) \leq 1.
\end{equation}

\noindent 
In 1979 Lagarias, Montgomery and Odlyzko gave an effective version of 
Chebotarev Density Theorem.

\begin{thm}[Lagarias, Montgomery and Odlyzko, 1979] \label{Lan}
There exists an effectively computable positive absolute constant $c_1$ such that 
for any number field $K$, any finite Galois extension $L/K$ and any conjugacy class
$C$ in $\Gal(L/K)$, there exists a prime $v$ of $K$ which is
unramified in $L$, for which  $\left(\ddfrac{L|K}{v}\right)=C$
and the norm $\textrm{N}_{K/\QQ}(v)$ is a rational prime
satisfying the bound

 \begin{equation} \label{norm} \textrm{N}_{K/\QQ}(v)\leq 2|D_{L/\QQ}|^{c_1}. \end{equation}

\end{thm}

\noindent In addition, in their recent paper \cite{AHN}, Ahn and Kwon show this more explicit result.

\begin{thm}[Ahn, Kwon, 2019] \label{AHN}
For any number field $K$, any finite Galois extension $L/K$, with $L\neq \QQ$ and any conjugacy class
$C$ in $\Gal(L/K)$, there exists a prime $v$ of $K$ which is
unramified in $L$, for which  $\left(\ddfrac{L|K}{v}\right)=C$
and the norm $\textrm{N}_{K/\QQ}(v)$ is a rational prime
satisfying the bound

 \begin{equation} \label{norm_2} \textrm{N}_{K/\QQ}(v)\leq |D_{L/\QQ}|^{12577}. \end{equation}

\end{thm}

\noindent Since $\QQ(\z_m)\subseteq K(\E[m])$, then $L=K(\E[m])\neq \QQ$,
for all $K$, $\E$ and $m\geq 3$. Thus we can apply Theorem
\ref{AHN} with $L=K(\E[m])$ (and in particular with $L=\QQ(\E[m])$), for every $m\geq 3$. 
It suffices that the local divisibility is satisfied for all  nonarchimedean places $v$ of $K$ with norm 
$N_{K/\QQ}(v)\leq |D_{L/\QQ}|^{12577}$, to have the
global one in the case when the Hasse principle for divisibility holds.  
Then we can take a
set  $S$ as  

$$\{v\in M_K\setminus\{\infty\}| h(N_{K/\QQ}(v))\leq 12577\cdot B(m,b,c)\}.$$

\noindent Moreover, as a consequence of inequality \eqref{density}, we can restrict this set still, as in \eqref{S}.

$$S=\{v\in M_K\setminus\{\infty\}| h(N_{K/\QQ}(v))\leq 12577 \cdot B(m,b,c)\}\setminus S'.$$ 

\begin{rem} \label{last_rem} 
Observe that when $K/\QQ$ is a Galois extension, then $N_{K/\QQ}(v)=\prod_{\sigma\in \Gal(K/\QQ)} \sigma(v)$
and by Proposition \ref{height}, the hypothesis that $h(N_{K/\QQ}(v))\leq 12577 \cdot B(m,b,c)$ in Problem 1' is
equal to $h(v)\leq 12577 \ddfrac{B(m,b,c)}{[K:\QQ]}.$ In particular, when $K=\QQ$,
we have $N_{K/\QQ}(v)=v$ and $h(v)=\log v$ and
it suffices to assume that $\log v\leq 12577\cdot B(m,b,c)$, as in Corollary \ref{cor_1}.
\end{rem}

\par\noindent Therefore, by the results produced about the local-global divisibility in elliptic curves over $\QQ$ mentioned in Section \ref{sec0}, together with Theorem \ref{AHN}, the inequality \eqref{density} and Remark \ref{last_rem}
the bound produced in Theorem \ref{DL_1} implies the following result.

\begin{cor} \label{cor_1} Let $\E: y^2=x^3+bx+c$ be an elliptic curve defined over $\QQ$ and let $L=\QQ(\E[m])$, where $m\geq 3$ is a fixed positive integer 
not divisible by any power $p^n$, with $p\in \{2,3\}$ and $n\geq 2$. Set 

$$  B(m,b,c):=  \left\{ \begin{array}{ll}  
5(m^2-1)^3(m^2-3)(\log m+h(b)+h(c)), & \textrm{if } m\geq 3  \textrm{ is odd; } \\
& \\
5(m^2-4)^3(m^2-6)(\log m+h(b)+h(c)),  & \textrm{if } m\geq 4  \textrm{ is even. } 
\end{array}
\right. 
$$

\noindent Let $P\in {\mathcal{E}}(\QQ)$ and assume that for all nonarchimedean places $v\in \QQ$, such that $h(v)\leq 12577\cdot B(m,b,c)$, but at most some of them with density $\delta < \ddfrac{1}{[L:\QQ]}$, there exists $D_v\in {\mathcal{E}}(\QQ_v)$ such that $P=mD_v$. Then there exists $D\in {\mathcal{E}}(\QQ)$ such that $P=mD$.
\end{cor}

\noindent As a consequence of Theorem \ref{DL_1}, together with Theorem \ref{AHN} and the inequality \eqref{density}, we can reformulate the statement of Problem \ref{prob1} as follows.

\begin{prob1}  Let $K$ be a number field, let $\E: y^2=x^3+bx+c$ be an elliptic curve defined over $K$ and let $L=K(\E[m])$, where $m\geq 3$ is a fixed positive integer. Set 

$$  B(m,b,c):=  \left\{ \begin{array}{ll}  
5(m^2-1)^3(m^2-3)(\log m+h(b)+h(c)), & \textrm{if } m\geq 3  \textrm{ is odd; } \\
& \\
5(m^2-4)^3(m^2-6)(\log m+h(b)+h(c)),  & \textrm{if } m\geq 4  \textrm{ is even. } 
\end{array}
\right. 
$$

\noindent Let $P\in {\mathcal{E}}(K)$. Assume that for all nonarchimedean places $v\in M_{K}$, such that $h(N_{K/\QQ}(v))\leq 12577 \cdot B(m,b,c)$, but at most some of them with density $\delta < \ddfrac{1}{[L:K]}$, there exists $D_v\in {\mathcal{E}}(K_v)$ such that $P=mD_v$. Is it possible to conclude that there exists $D\in {\mathcal{E}}(K)$ such that $P=mD$?
\end{prob1}

\section{An example} \label{existence}

\par In this section we will produce an example of an elliptic curve in Weiestrass form 

\begin{equation} \label{E} \E: y^2=(x-\alpha)(x-\beta)(x-\gamma), \end{equation}

\noindent where $\al,\be,\ga\in\QQ$ and $\alpha+\be+\ga=0$, with a point $P$, locally divisible by $4$ 
for infinitely many primes but not globally divisible by $4$. 
Similar examples have been given in \cite{DZ2} and in \cite{Pal} for curves such that the Hasse principle for divisibility does not work, but here we instead consider a curve for which the local-global divisibility by $4$ holds and give a method to find points failing the hypotheses of the 
local-global principle even if they are locally divisible for infinitely many places. 
For what we have discussed in the previous sections, it suffices that the local divisibility holds for a finite number of places
that should be distribuited with a certain density; in our example we have the local divisibility for infinitely many primes but they
are not distribuited with the necessary density. 
In particular the density of prime numbers for which the local divisibility fails is indeed $\delta > \ddfrac{1}{[\QQ(\E[4]):\QQ]}$ (contradicting the hypotheses of Problem 1'). We call those points pseudodivisible since they apparently satisfy the
hypotheses of the local-global principle for divisibility (being locally divisible for infinitely many places), but indeed they fail them
(and then they fail the local-global principle).
 Let $G_p$ be the $p$-Sylow subgroup of the image of the representation of the Galois group $G=\Gal(\QQ(K(E[4])/\QQ)$
in $\GL_2(\ZZ/4\ZZ)$ and let $G_0$ be the kernel of the reduction modulo $2$ of the matrices in $G$. 
Let $G$ be the subgroup of $\GL_2(\FF_2)$ generated by the matrices

\bigskip

 $\s_1=\left(%
\begin{array}{cc}
  -1 & \hspace{0.3cm} 0 \\
  \hspace{0.3cm} 2 & -1 \\
\end{array}%
\right),$    $\s_2=\left(%
\begin{array}{cc}
  1 & \hspace{0.3cm} 2 \\
  2 & -1 \\
\end{array}%
\right),$ 
 $\s_3=\left(%
\begin{array}{cc}
  1 & 0 \\
  2 & 1 \\
\end{array}%
\right),$ \hspace{0.2cm}
$\s_4=\left(%
\begin{array}{cc}
  -1 &  0 \\
 \hspace{0.3cm} 0 & 1 \\
\end{array}%
\right)$.

\medskip\noindent Thus

$$G=\left\langle \sigma(x,y,z,w)\left | \hspace{0.1cm} \Id+2\left(%
\begin{array}{cc}
  x+w & y \\
  x+y+z & x+y \\
\end{array}%
\right), \right. \textrm{with } x,y,z,w\in \ZZ/4\ZZ \right\rangle. $$

\noindent We have $\sigma_1=\sigma(1,0,0,0), \sigma_2=\sigma(0,1,0,0),
\sigma_3=\sigma(0,0,1,0), \sigma_4=\sigma(0,0,0,1)$. 
 One can easily verify that
$G\simeq (\FF_2) ^4=G_p\cap G_0$. In particular $G_0\cap G_p$ has dimension $4$
as vector space over $\FF_2$ and then, by \cite[Proposition 3.2, Case (iii)]{DZ}, 
 we have that the local-global divisibility by
$4$ holds in $\E$ over $\QQ$. Thus, by the results produced in Section \ref{sec1},
if the local divisibility holds for every nonarchimedean place $v\in M_{\QQ}$ such that

$$\log v \leq 5\cdot 12^3\cdot 10\cdot (\log 4+h(\alpha\beta+\beta\gamma+\alpha\gamma)+
h(\alpha\beta\gamma)),$$

\noindent but at most some of them with density $\delta < \ddfrac{1}{16}$, then the global divisibility
holds as well. Of course, for explicit nonzero integers $\alpha\beta+\beta\gamma+\alpha\gamma$
 and $\alpha\beta\gamma$, such bound 
can be written in a more elegant form as in Remark \ref{rem_eleg}. 
Let $\{Z_{\s}\}_{\s\in G}$ be the cocycle of $G$ with values in $\E[4]$, defined by $$Z_{\s(x,y,z,w)}:=\left(%
\begin{array}{c}
  2w \\
  0 \\
\end{array}%
\right).$$

\noindent For $i\in \{1,2,3,4\}$ we have that $Z_{\s_i}$ satisfies the local conditions. 
Anyway there are some $\sigma\in G$ such that
$Z_{\s}$ does not satisfy the local conditions. For instance

$$\s(1,0,0,1)=\left(%
\begin{array}{cc}
  1 & \hspace{0.3cm} 0 \\
 2 & -1 \\
\end{array}%
\right)$$

\noindent and

$$(\s(1,0,0,1)-1)\left(%
\begin{array}{c}
  \alpha \\
  \beta \\
\end{array}%
\right)=\left(%
\begin{array}{cc}
  0 &  0 \\
  2 & 2 \\
\end{array}%
\right)\left(%
\begin{array}{c}
  \alpha \\
  \beta \\
\end{array}%
\right)\neq \left( \begin{array}{c}
  2 \\
  0 \\
\end{array}%
\right),$$

\noindent for every $(\alpha,\beta)\in \ZZ/4\ZZ.$ Thus $\s(1,0,0,1)$ does not satisfy
the local conditions and $\{\Z_{\s}\}_{\s\in G}$ does not define
a class in $H^1_{\loc}(G,\E[4])$. In a similar way one can verify that $Z_{\s(0,1,0,1)}, Z_{\s(0,0,1,1)}$, and $Z_{\s(1,0,1,1)}$
are the only other images of the cocycle $Z$ that do not
satisfy the local conditions.
In any case, since $Z_{\s_i}$ satisfies the local conditions, for every $i\in \{1,2,3,4\}$, then we
search for a point $D\in \E(\overline{\QQ})$, such that $Z_{\s_i}=\s_i(D)-D$, for every $i\in \{1,2,3,4\}$.
We will find a point $P=4D$ which is locally divisible by $4$ in infinitely many $p$-adic fields (this is
assured by the validity of the local conditions for $\s_i$, with $i=1, ..., 4$). 
Anyway, we expect that the local-global divisibility should fail
for such a point $P$ arising from this cocycle, even if $P$ is locally divisible for an infinite number
of primes. We will show that this is indeed the case. 
Notice that the density of primes for which the local divisibility fails is about  $\ddfrac{4}{16}=0,25$,
since the local conditions are not satisfied only by the following 4 images of the cocycle $Z$: 
$$Z_{\s(1,0,0,1)},Z_{\s(0,1,0,1)}, Z_{\s(0,0,1,1)} \textrm{ and } Z_{\s(1,0,1,1)}.$$

\noindent A generating set of $\QQ(\E[4])$ is given by the points
$$A'=(\alpha+\sqrt{(\alpha-\beta)(\alpha-\gamma)},
(\alpha-\beta)\sqrt{\alpha-\gamma}+(\alpha-\gamma)\sqrt{\alpha-\beta}),$$ 
$$B'=(\beta+\sqrt{(\beta-\alpha)(\beta-\gamma)},(\beta-\alpha)\sqrt{\beta-\gamma}+(\beta-\gamma)\sqrt{\beta-\alpha}),$$ 

\noindent with $2A'=A=(\alpha,0)$ and $2B'=B=(\beta,0)$  \cite{DZ2}, \cite{Pal}. To find
a suitable elliptic curve $\E$, with $G=\langle\s_1,\s_2,\s_3,\s_4 \rangle$,
we require that the columns of $\s_i$ are $\s_i(A')$ and $\s_i(B')$, for
$i=1, ..., 4$. Thus we request $\sigma_1(A')=-A'+2B'=B-A'$ and
$\sigma_2(B')=-B'$. By a bit of calculations, one can see that
$B-A'=(\alpha-\sqrt{(\alpha-\beta)(\alpha-\gamma)},(\alpha-\beta)\sqrt{\alpha-\gamma}-(\alpha-\gamma)\sqrt{\alpha-\beta})$.
We deduce

\[
\begin{array}{llll}
\textbf{1)} & \s_1(\sqrt{\al-\be})=-\sqrt{\al-\be};
&\textbf{3)} & \s_1(\sqrt{-1})=\sqrt{-1};\\
\textbf{2)} & \s_1(\sqrt{\be-\ga})=-\sqrt{\be-\ga};
& \textbf{4)} & \s_1(\sqrt{\al-\ga})=\sqrt{\al-\ga}.\\
\end{array} \]

\noindent Furthermore we should have $\s_2(A')=A'+2B'=A'+B$ and
$\s_2(B')=2A'-B'=A-B'$. Observe that $-B=B$ and $A'+B=-(-B-A')=-(B-A')=(\alpha-\sqrt{(\alpha-\beta)(\alpha-\gamma)},-(\alpha-\beta)\sqrt{\alpha-\gamma}+(\alpha-\gamma)\sqrt{\alpha-\beta})$.
Moreover one can verify that 
$A-B'=(\beta-\sqrt{(\beta-\alpha)(\beta-\gamma)},(\beta-\alpha)\sqrt{\beta-\gamma}-(\beta-\gamma)\sqrt{\beta-\alpha})$. We deduce

\[
\begin{array}{llll}
\textbf{5)} & \s_2(\sqrt{\al-\be})=\sqrt{\al-\be};
&\textbf{7)} & \s_2(\sqrt{-1})=-\sqrt{-1};\\
\textbf{6)} & \s_2(\sqrt{\be-\ga})=\sqrt{\be-\ga};
& \textbf{8)} & \s_2(\sqrt{\al-\ga})=-\sqrt{\al-\ga}.\\
\end{array} \]

\noindent Regarding $\s_3$, we require
$\s_3(A')=A'+2B'=A'+B$ and $\s_3(B')=B'$, i. e.

\[
\begin{array}{llll}
\textbf{9)} & \s_3(\sqrt{\al-\be})=\sqrt{\al-\be};
&\textbf{11)} & \s_3(\sqrt{-1})=\sqrt{-1};\\
\textbf{10)} & \s_3(\sqrt{\be-\ga})=\sqrt{\be-\ga};
& \textbf{12)} & \s_3(\sqrt{\al-\ga})=-\sqrt{\al-\ga}.\\
\end{array} \]

\noindent Finally we should have $\s_4(A')=-A'$ and $\s_4(B')=B'$, implying

\[
\begin{array}{llll}
\textbf{13)} & \s_4(\sqrt{\al-\be})=-\sqrt{\al-\be};
&\textbf{15)} & \s_4(\sqrt{-1})=-\sqrt{-1};\\
\textbf{14)} & \s_4(\sqrt{\be-\ga})=\sqrt{\be-\ga};
& \textbf{16)} & \s_4(\sqrt{\al-\ga})=-\sqrt{\al-\ga}.\\
\end{array} \]

\noindent Since $$Z_{\s_i}=\left(%
\begin{array}{c}
  0 \\
  0 \\
\end{array}%
\right),$$

\noindent for every $i\in \{1,2,3\}$, then $\sigma_i(D)=D$, for every $i\in \{1,2,3\}$.
Thus $D\in \E(K_4^{\langle \s_1,\s_2,\s_3\rangle}),$ where
$K_4^{\langle \s_1,\s_2,\s_3\rangle}$ denotes the field fixed by $\langle \s_1,\s_2,\s_3\rangle$. We have 

$$K_4^{\langle \s_1,\s_2,\s_3\rangle}=\QQ\left(\sqrt{(\al-\be)(\be-\ga)}\right).$$ 

\noindent Let $(\al-\be)(\be-\ga)=\delta \omega^2$,
with $\delta,\omega\in \QQ$ and $\delta$ squarefree. Therefore $D=(u,v)$, with
$u=u_0+u_1\sqrt{\delta}$ and $v=v_0+v_1\sqrt{\delta}$.
By a computation showed in \cite{DZ2}, every such point $D$ 
corresponds to a point $(s,t)$ satisfying the equation

\begin{equation} \label{B} {\mathcal{B}}: \delta s^2=\delta^2t^4-6\alpha\delta t^2-(\be-\ga)^2, \end{equation}

\noindent by $u_1=s/2$, $u_0=(t^2\delta-\alpha)/2$ and $v=t\sqrt{\delta}(u-\alpha)$. Now we have to choose $\alpha, \beta, \gamma$ such that $[K_4:\QQ]=16$.
We set $\alpha=9$, $\be=6$ and $\ga=-15$. Then $\al-\be =3$, $\be-\ga=21$,
$\al-\ga=24$ and $K_4=\QQ(\sqrt{2},\sqrt{3},\sqrt{-1},\sqrt{7})$. Moreover $\delta =7$.
Thus the curve $\E$ has equation

$$y^2=x^3-171x+810$$

\noindent and \eqref{B} becomes

$${\mathcal{B}}: s^2=7t^4-54t^2+63.$$

\noindent A rational point of $\mathcal{B}$ is $(s,t)=(4,1)$
(other rational points are for instance $(12,3)$ and $(204,9)$). The point  $(s,t)=(4,1)$ corresponds to
$D=(-1+2\sqrt{7},14-10\sqrt{7})$. Observe that $D^{\s_4}-D=(-1-2\sqrt{7},14+10\sqrt{7})+(-1+2\sqrt{7},-14+10\sqrt{7})
=A,$ i.e.

$$Z_{\s_4}=\left(%
\begin{array}{c}
  2 \\
  0 \\
\end{array}%
\right),$$

\noindent as expected. We have  $P=4D=(10,10)$.
By the software of computational algebra AXIOM (that is also implemented in SAGE), we calculated all
the $4$-divisor of $P$. The abscissas of the $4$-divisors of $P$ are the
roots of the polynomial

\[\begin{split}
\varphi_4(x):=  \hspace{0.3cm} &  x^{16}   - 160x^{15}   + 6840x^{14}   - 139680x^{13}   + 4862268x^{12}- 134693280x^{11} \\
&+2294454600x^{10} - 32425103520x^9  + 300976938918x^8\\
&  + 1203164578080x^7  - 68296345025400x^6 + 695993396274720x^5  \\
&- 1996085493644292x^4 -14987477917513440x^3   + 146812808536034040x^2 \\
& - 478587272134802400x +570463955816032161,\\
\end{split} \]

\bigskip \noindent which on $K_4$ factors as

\[
\begin{split}
& (x + 2\sqrt{7} + 1)(x - 2\sqrt{7} + 1)(x + 6\sqrt{7} - 27) (x - 6\sqrt{7} - 27)
  \cdot \\
   &   (x + 6\sqrt{-3} + 6\sqrt{3} + 12\sqrt{-1} + 3) (x - 6\sqrt{-3} - 6\sqrt{3} + 12\sqrt{-1} + 3)
  \cdot \\
  &    (x + 6\sqrt{-3} - 6\sqrt{3} - 12\sqrt{-1} + 3) (x - 6\sqrt{-3} + 6\sqrt{3} - 12\sqrt{-1}+ 3)
  \cdot \\
&    (x - 3\sqrt{-6} + 6\sqrt{-3} - 3\sqrt{2} - 3) (x + 3\sqrt{-6} - 6\sqrt{-3} - 3\sqrt{2} - 3)
  \cdot \\
  &   (x - 3\sqrt{-6} - 6\sqrt{-3} + 3\sqrt{2} - 3)(x + 3\sqrt{-6}  + 6\sqrt{-3} + 3\sqrt{2} - 3)
  \cdot \\
   & (x - 3\sqrt{42} - 6\sqrt{21} - 21\sqrt{2} - 27)(x + 3\sqrt{42} + 6\sqrt{21}  - 21\sqrt{2} - 27)
  \cdot \\
    & (x - 3\sqrt{42} + 6\sqrt{21} + 21\sqrt{2} - 27) (x + 3\sqrt{42} - 6\sqrt{21}  + 21\sqrt{2} - 27)
\end{split} \]

\bigskip\noindent The 16 abscissas of the 4-divisors of $P$ are the following:

\[
\begin{array}{ll}
  x_1 =-1+2\sqrt{7}; & x_2=-1-2\sqrt{7};\\
  x_3  =27-6\sqrt{7}; & x_4=27+6\sqrt{7};\\
  x_5  =-3+6\sqrt{-3}+6\sqrt{3}-12\sqrt{-1};  & x_6=6\sqrt{-3}-6\sqrt{3}+12\sqrt{-1}-3;\\
  x_7  =-3-6\sqrt{-3} + 6\sqrt{3} + 12\sqrt{-1};& x_8=-3-6\sqrt{-3} - 6\sqrt{3} - 12\sqrt{-1};\\
  x_9  =3+3\sqrt{-6} + 6\sqrt{-3} - 3\sqrt{2}; & x_{10}=3+3\sqrt{-6} - 6\sqrt{-3} + 3\sqrt{2};\\
 x_{11}  =3- 3\sqrt{-6} + 6\sqrt{-3} + 3\sqrt{2};  & x_{12}=3- 3\sqrt{-6}- 6\sqrt{-3} - 3\sqrt{2};\\
  x_{13}  =27- 3\sqrt{42} + 6\sqrt{21} - 21\sqrt{2};  & x_{14}=27- 3\sqrt{42} - 6\sqrt{21} + 21\sqrt{2};\\
  x_{15}  =27+3\sqrt{42} + 6\sqrt{21} + 21\sqrt{2};  & x_{16}=27+ 3\sqrt{42}  - 6\sqrt{21}  - 21\sqrt{2};\\
\end{array} \]

\bigskip\noindent and the ordinates corresponding to $x_{i}$, for every $1\leq i\leq 16$ are respectively the following

\[
\begin{array}{ll}
y(x_1) = \pm (14-10\sqrt{7}); & y(x_2)= \pm (14+10\sqrt{7});\\
y(x_3) =  \pm (126-54\sqrt{7}); & y(x_4) =\pm (126+54\sqrt{7});\\
y(x_5) =  \pm (72-6\sqrt{-3}- 42\sqrt{3}); & y(x_6) =\pm  (72-6\sqrt{-3}+42\sqrt{3});\\
 y(x_7) = \pm  (72 + 6\sqrt{-3} - 42\sqrt{3}); &y(x_8) =\pm  (72 + 6\sqrt{-3}  + 42\sqrt{3});\\
 y(x_9) = \pm (36+ 27\sqrt{2} -15\sqrt{-6} -12\sqrt{-3}); & y(x_{10}) =\pm (36- 27\sqrt{2} -15\sqrt{-6}+12\sqrt{-3});\\
y(x_{11}) = \pm (36- 27\sqrt{2} +15\sqrt{-6}-12\sqrt{-3});  &y(x_{12}) =\pm (36+27\sqrt{2} +15\sqrt{-6}+12\sqrt{-3});\\
y(x_{13}) =  \pm (252- 39\sqrt{42}  + 60\sqrt{21}  - 189\sqrt{2}); &y(x_{14}) =\pm  (252- 39\sqrt{42} - 60\sqrt{21} + 189\sqrt{2});\\
y(x_{15}) =  \pm  (252+39 \sqrt{42}+ 60\sqrt{21} + 189\sqrt{2}); &y(x_{16}) =\pm  (252+39\sqrt{42} - 60\sqrt{21} - 189\sqrt{2}).\\
\end{array} \]

\bigskip\noindent  Thus

\begin{description}
\item[i)] four among the $4$-divisors of $P$  have coordinates in $\QQ(\sqrt{7})$;
\item[ii)] four  among the $4$-divisors of $P$ have coordinates in $\QQ(\sqrt{-1},\sqrt{-3})$;
\item[iii)] four among the $4$-divisors of $P$ have coordinates in $\QQ(\sqrt{-3},\sqrt{2})$;
\item[iv)] four among the $4$-divisors of $P$ have coordinates in $\QQ(\sqrt{2},\sqrt{21})$.
\end{description}

\noindent None of the $4$-divisors of  $P$ has coordinates in $\QQ$. Thus $P$ is not globally divisible by $4$ over $\QQ$. \par 
Anyway, as stated above, $P$ is divisible by $4$ in $\QQ_p$ for infinitely many prime numbers $p$. 
Just to have an idea of the distribution of primes for which the locally divisibility does not hold we will describe the situation for all $p< 1000$. 
To know if a $4$-divisor of $P$ has
coordinates in $\QQ_p$, for any prime number $p$, we can use the quadratic reciprocity law or we can factor
$\varphi_4(x)$ on $\QQ_p$, by the use of a software of computational algebra.
By using the software PARI, we verified that the equation $\varphi_4(x)=0$ has a solution in $\QQ_p$ for the following 123 primes $p<1000$:

\bigskip \noindent  3, 7, 13, 17, 19, 29, 31, 37, 41, 47, 53, 59, 61, 73, 79, 83, 89, 97, 103, 109, 113, 127, 131, 137, 139, 149, 151, 157, 167, 181, 193, 197, 199, 223, 227, 229, 233, 241, 251, 257, 271, 277, 281, 283, 307, 311, 313, 317, 337, 349, 353, 367, 373, 383, 389, 397, 401, 409, 419, 421, 433, 439, 449, 457, 463, 467, 479, 487, 503, 521, 523, 541, 557, 563, 569, 577, 587, 593, 601, 607, 613, 617, 619, 631, 641, 643, 647, 653, 661, 673, 691, 701, 709, 719, 727, 733, 751, 757, 761, 769, 787, 809, 811, 821, 823, 829, 839, 853, 857, 859, 877, 881, 887, 919, 929, 937, 953, 967, 971, 977, 983, 991, 997.

\bigskip\noindent In the same way we verified that instead the equation $\varphi_4(x)=0$ has no solution in $\QQ_p$ for the following 45 primes $p<1000$:

\bigskip \noindent 2, 5, 11, 23, 43, 67, 71, 101, 107, 163, 173, 179, 191, 211, 239, 263, 269, 293, 331, 347, 359, 379, 431, 443, 461, 491, 499, 509, 547, 571, 599, 659, 677, 683, 739, 743, 773, 797, 827, 863, 883, 907, 911, 941, 947.

\bigskip\noindent Notice that the density of primes $p< 1000$ for which we have no solution is $\ddfrac{45}{168}\sim 0,26$. 
 Even if calculated just for a finite number of primes, instead of all but finitely many primes, this density
corresponds to the expected density $\ddfrac{4}{16}=0,25$ calculated by
the number of cocycles not satisfying the local conditions and it is greater
than $\ddfrac{1}{|G|}=\ddfrac{1}{16}= 0,0625$, which is the maximum density required
by the hypotheses Problem 1'.

\bigskip\noindent \textbf{Acknowledgments}. 
The idea of investigating about pseudodivisible points and giving an explicit effective version of the hypotheses of the local-global divisibility in elliptic curves over $\QQ$ was originally conceived when Laura Paladino and
Igor Shparlinski were both guests at the Max Planck Institute for Mathematics in Bonn.
The authors are grateful to Igor Shparlinski for suggesting the
idea of studying these topics and for useful discussions. Laura Paladino is
also grateful to MPIM for the hospitality and the excellent work conditions. The authors thank Joseph H. Silverman for some suggestions about the introduction of this paper.

\end{document}